\documentclass[10.5pt]{amsart}
\usepackage{yfonts}
\usepackage[all]{xy}
\usepackage{amsfonts,amscd,amssymb,amsthm}
\usepackage{amsmath,amsxtra,eucal,graphicx,graphics}
\usepackage{tikz,float}
\usepackage{enumerate}
 \usepackage[T1]{fontenc}
\usepackage[utf8]{inputenc}
\newcommand{\overbar}[1]{\mkern 1.5mu\overline{\mkern-1.5mu#1\mkern-1.5mu}\mkern 1.5mu}

\theoremstyle{definition}
\newtheorem{theorem}{Theorem}[section]

\newtheorem{lemma}[theorem]{Lemma}
\newtheorem{corollary}[theorem]{Corollary}

 \title{Maximal distance spectral radius of $4$-chromatic planar graphs}
 
\author{Aysel Erey}
\address{Department of Mathematics\\
Gebze Technical University\\
Kocaeli, Turkey}
\email{aysel.erey@gtu.edu.tr}

 \date{\today}
  
\begin{document}

\begin{abstract}
We show that the kite graph $K_4^{(n)}$ uniquely maximizes the distance spectral radius among all connected $4$-chromatic planar graphs on $n$ vertices.
\end{abstract}

\keywords{chromatic number, planar graphs, distance spectral radius}
\subjclass[2010]{05C50}

 \maketitle
 
\section{Introduction}
In this article all graphs are finite, simple and undirected.  The {\it distance} between two vertices $u$ and $v$ of a connected graph $G$, denoted by $d_G(u,v)$, is the size of a shortest path between $u$ and $v$ in $G$. Let $v_1,\dots ,v_n$ be the vertices of a connected graph $G$. The {\it distance matrix} $D(G)$ of $G$ is an $n$ by $n$ symmetric matrix given by
$$(D(G))_{ij}=d_G(v_i,v_j).$$

For the rest, we shall assume that all mentioned graphs are connected whenever the distance matrices are concerned. Since $D(G)$ is a symmetric matrix, all of its eigenvalues are real. By the Perron-Frobenius Theorem, the largest eigenvalue of $D(G)$ is positive and it has multiplicity one.  The {\it distance spectral radius} $\rho (G)$ of $G$ is the largest eigenvalue of $D(G)$. One of the central problems in the recent study of distance matrices is finding extremal graphs maximizing or minimizing the distance spectral radius in a given family of graphs, see \cite{survey} for a recent survey on the subject. One of the first results in this direction is due to \cite{path} where it was shown that the path graph $P_n$ uniquely maximizes the distance spectral radius in the family of graphs on $n$ vertices. In \cite{cacti}, the authors determine the extremal values of the distance spectral radius in the family of cacti with $n$ vertices and $k$ cycles. Extremal values of $\rho (G)$ were also determined in \cite{graft} when $G$ is a graph on $n$ vertices with $k$ pendant vertices. Maximal distance spectral radius of a graph was determined in various classes of  trees such as trees with a fixed maximum degree \cite{stevan}, trees on $n$ vertices and matching number $m$ \cite{nath_paul} and trees with given order and number of segments \cite{tree_2020}. Minimal distance spectral radius of a graph was determined in graphs on $n$ vertices with $k$ cut vertices (or $k$ cut edges) \cite{godsil} and in multipartite graphs of order $n$ with $t$ parts \cite{oboudi}.

 While  relations between  chromaticity of graphs and the spectra of other graph matrices such as Laplacian or adjacency matrices have been extensively studied in the literature, very few results are known on  relations between chromaticity and distance spectra of graphs. Let  $\mathcal{G}_{k,n}$ denote the family of connected $k$-chromatic graphs of order $n$. To minimize $\rho (G)$ of a graph $G$ in $\mathcal{G}_{k,n}$, it suffices to consider only complete $k$-partite graphs on $n$ vertices, thanks to Perron-Frobenius Theorem. Using this fact, it was shown in \cite{chromDist} that the Tur{\'a}n graph uniquely minimizes distance spectral radius in $\mathcal{G}_{k,n}$. On the other hand, the problem of maximizing the distance spectral radius of a graph in $\mathcal{G}_{k,n}$ is wide open. In this article, we make a contribution to this problem by studying the family of $4$-chromatic planar graphs which is one of the most interesting and challenging graphs families in the field of chromatic graph theory. The famous Four Color Theorem which stood as an unsolved problem for over a century says that every planar graph is $4$-colorable  \cite{appel,seymour}.

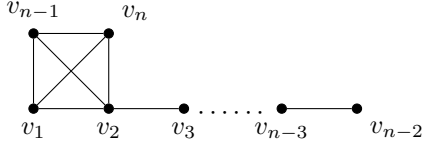
\begin{figure}[h!]
\begin{center}
\begin{tikzpicture}
[scale=1, vertices/.style={draw, fill=black, circle, minimum size = 10pt, inner sep=2pt}, another/.style={draw, fill=black, circle, minimum size = 3.5pt, inner sep=0.1pt}]
\node[another, label=below:{$v_1$}] (v1) at (0,0) {};
\node[another, label=below:{$v_2$}] (v2) at (1,0) {};
\node[another, label=below:{$v_3$}] (v3) at (2,0) {};
\node[another, label=above:{$v_{n-1}$}] (a) at (0,1) {};
\node[another, label=above right:{$v_{n}$}] (b) at (1,1) {};
\node[another, label=below:{$v_{n-3}$}] (c) at (3.3,0) {};
\node[another, label=below right:{$v_{n-2}$}] (d) at (4.3,0) {};
\foreach \to/\from in {v1/v2, v2/v3, v1/a, v1/b, v2/a, v2/b, a/b, c/d}
\draw [-] (\to)--(\from);
\node[label=above:{$\cdots \cdots$}](1) at (2.65,-.38){};
\end{tikzpicture}
\caption{The kite graph $K_4^{(n)}$.}
\label{kite_picture}
\end{center}
\end{figure}

The \textit{kite graph} $K_{k}^{(n)}$ is the graph of order $n$ obtained from a $k$-clique by attaching a path of length $n-k$  at any vertex of the $k$-clique (see Fig.~\ref{kite_picture} for $k=4$). Kite graphs are known to be extremal graphs for various graph parameters. For example, in \cite{clique} it was shown that $K_{k}^{(n)}$ is the unique extremal graph with maximum distance spectral radius among graphs with fixed clique number $k$ and order $n$. Moreover, it is known that $K_k^{(n)}$ has the largest Wiener index in $\mathcal{G}_{k,n}$  and it minimizes the adjacency spectral radius of a graph in $\mathcal{G}_{k,n}$ with $k\geq 4$, see \cite{kite_Wiener} and \cite {chromADJ} respectively. It is easy to see that $K_3^{(n)}$ is the unique graph maximizing the distance spectral radius of a graph in $\mathcal{G}_{3,n}$ (see Corollary~\ref{3chromatic}). In this article, we focus on $4$-chromatic planar graphs and our main result is the following:

\begin{theorem}\label{main}
The kite graph $K_4^{(n)}$ is the unique graph maximizing the distance spectral radius among all connected $4$-chromatic planar graphs of order $n$.
\end{theorem}

\section{Graph Theory Terminology}

Let $G$ be a graph with vertex set $V(G)$ and edge set $E(G)$. The {\it order} of $G$ is $|V(G)|$ and its {\it size} is $|E(G)|$. We write $P_n$, $K_n$ and $C_n$ for the path, complete and cycle graphs on $n$ vertices respectively. The graph $K_1$ is called the {\it trivial graph}. We denote a cycle graph $C_n$ by $v_1v_2\cdots v_n$ if $E(C_n)=\{v_iv_{i+1}: 1\leq i\leq n-1\}\cup\{v_1v_n\}$. A {\it k-clique} of a graph $G$ is a subgraph isomorphic to the complete graph $K_k$ and the {\it clique number} of $G$ is the order of the largest clique in $G$. A subset of vertices $S\subseteq V(G)$ is called an {\it independent set} of $G$ if no two vertices in $S$ are adjacent in $G$. The degree of a vertex $v$ in $G$ is denoted by $deg_G(v)$ and the open and closed neighborhoods of $v$ in $G$ are denoted by $N_G(v)$ and $N_G[v]$ respectively. For a subset of vertices $S$ in $G$, the vertex subset $N_G(S)$ consists of all vertices in $G$ which has at least one neighbor in $S$. Also, we define $N_G[S]=S\cup N_G(S)$. We say that $v$ is a {\it degree k vertex} in $G$ if $deg_G(v)=k$. We write $\Delta (G)$ and $\delta (G)$ for the maximum and minimum degrees of $G$ respectively. An {\it isolated vertex} of a graph $G$ is a vertex which has no neighbors in $G$. A {\it leaf vertex} is a vertex of degree $1$ and a {\it pendant edge} is an edge which contains a leaf vertex. Two vertices $u$ and $v$ in $G$ are called {\it twin vertices} if either $N_G(u)=N_G(v)$ or $N_G[u]=N_G[v]$. Let $\overline{G}$ denote the complement of $G$ and $rG$ denote the disjoint union of $r$ copies of $G$. The {\it join} of $G$ and $H$, denoted by $G\vee H$, is the graph obtained from the disjoint union of $G$ and $H$ by adding all edges $uv$ where $u\in V(G)$ and $v\in V(H)$. We say that a graph $G'$ is obtained from $G$ by attaching $H$ to $G$ at $v$ if $G'=G\cup H$ and $V(G)\cap V(H)=\{v\}$. If $H$ is a subgraph of $G$, we write $G\setminus H$ for the subgraph of $G$ induced by $V(G)\setminus V(H)$.
A {\it proper $k$-coloring} of a graph $G$ is a function $f:V(G)\rightarrow \{1,\dots , k\}$ such that $f(u)\neq f(v)$ for every edge $uv\in E(G)$. The {\it chromatic number} $\chi (G)$ of a graph $G$ is the minimum integer $k$ such that $G$ has a proper $k$-coloring. A graph $G$ is called {\it k-chromatic} if $\chi (G)=k$ and $G$ is called {\it k-colorable} if $\chi (G)\leq k$. We say that $G$ is a {\it k-critical} graph if $G$ is $k$-chromatic and $\chi (H)< \chi (G)$ for every proper subgraph $H$ of $G$. It is well known that if $G$ is $k$-critical, then $\delta (G)\geq k-1$. A graph is called {\it planar} if it can be drawn on the plane such that no two edges cross each other. A {\it plane} graph is a drawing of a planar graph on the plane such that no two edges cross each other.

Let $G_1,\dots , G_k$ be subgraphs of $G$. We say that $G_1,\dots , G_k$ are {\it vertex (edge) disjoint} if no two of them have a common vertex (edge). We say that $G_1,G_2, G_3$ form three {\it cactus-type cycles} in $G$ if they are edge disjoint cycles and the edge-induced subgraph of $G$ induced by $E(G_1)\cup E(G_2)\cup E(G_3)$ has exactly three cycles in it. A connected graph $G$ is called a \textit{cactus}  if every two cycles in $G$ have at most one common vertex. Let $\mathcal{C}(n,k)$ denote the family of cacti of order $n$ having exactly $k$ cycles.

\section{Preliminaries}

The following result is an immediate consequence of the Perron-Frobenius Theorem.

\begin{theorem}\label{subgraph}
	Let $e$ be an edge of a connected graph $G$ and suppose that \ $G\setminus e$ is also connected. Then,
	$\rho(G\setminus e)>\rho(G).$
\end{theorem}

\begin{lemma}\cite{stevan}\label{stevan} Let $v$ be a vertex of a nontrivial graph $G$. For $k,l\geq 0$, we denote by $G(v,k,l)$ the graph obtained from disjoint unions of $G$, $P_k$ and $P_l$ by adding edges between $v$ and one of the leaf vertices in both $P_k$ and $P_l$. If $k\geq l\geq 1$, then
$$\rho (G(v,k+1,l-1))>\rho (G(v,k,l)).$$
\end{lemma}

\begin{lemma}\cite{godsil}\label{godsil}
Let $u$ and $v$ be two adjacent vertices of a connected graph $G$ where $|E(G)|\geq 2$. For positive integers $k$ and $l$, let $G_{k,l}$ denote the graph obtained from $G$ by attaching paths of length $k$ at $u$ and length $l$ at $v$ with $u$ and $v$ being leaf vertices of the corresponding paths. If $k>l\geq 1$, then $\rho(G_{k,l})<\rho(G_{k+1,l-1})$; if $k=l\geq 1$, then $\rho (G_{k,l})<\rho (G_{k+1, l-1})$ or $\rho (G_{k,l})<\rho(G_{k-1,l+1})$.
\end{lemma}

The \textit{Saw graph} $S(p,q;l)$ is the cactus graph of order $2p+2q+l+1$ and size $3p+3q+l$ which is obtained from a path graph on $p+q+l$ edges by replacing $p$ consecutive edges including one pendant edge by $p$ triangles and replacing $q$ consecutive edges including the other pendant edge by $q$ triangles, see Fig~\ref{saw_picture}. The following result follows from the proof of Theorem 5.3 in \cite{cacti}.

\begin{theorem}\cite{cacti}\label{cacti}
If $G$ is a graph with maximal distance spectral radius in $\mathcal{C}(n,k)$, then $G\cong S(p,q;l)$ where $p+q=k$ and $l=n-2k-1$.
\end{theorem}

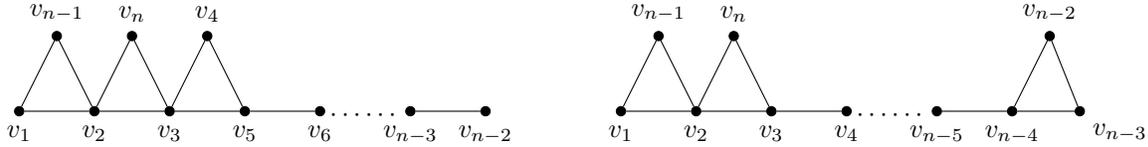
\begin{figure}[h!]
\begin{tikzpicture}
[scale=1, vertices/.style={draw, fill=black, circle, minimum size = 10pt, inner sep=2pt}, another/.style={draw, fill=black, circle, minimum size = 3.5pt, inner sep=0.1pt}]
\node[another, label=below:{$v_1$}] (v1) at (0,0) {};
\node[another, label=below:{$v_2$}] (v2) at (1,0) {};
\node[another, label=below:{$v_3$}] (v3) at (2,0) {};
\node[another, label=below:{$v_5$}] (v5) at (3,0) {};
\node[another, label=below:{$v_6$}] (v6) at (4,0) {};
\node[another, label=above:{$v_{n-1}$}] (a) at (0.5,1) {};
\node[another, label=above:{$v_{n}$}] (b) at (1.5,1) {};
\node[another, label=above:{$v_{4}$}] (v4) at (2.5,1) {};
\node[another, label=below:{$v_{n-3}$}] (c) at (5.2,0) {};
\node[another, label=below:{$v_{n-2}$}] (d) at (6.2,0) {};
\foreach \to/\from in {v1/v2, v2/v3, v1/a, v3/b, v2/a, v2/b, c/d, v3/v5, v4/v3, v4/v5, v5/v6}
\draw [-] (\to)--(\from);
\node[label=above:{$\cdots \cdots$}](1) at (4.6,-.38){};
\node[another, label=below:{$v_1$}] (u1) at (8,0) {};
\node[another, label=below:{$v_2$}] (u2) at (9,0) {};
\node[another, label=below:{$v_3$}] (u3) at (10,0) {};
\node[another, label=below:{$v_4$}] (u4) at (11,0) {};
\node[another, label=below:{$v_{n-5}$}] (n5) at (12.2,0) {};
\node[another, label=above:{$v_{n-1}$}] (a') at (8.5,1) {};
\node[another, label=above:{$v_{n}$}] (b') at (9.5,1) {};
\node[another, label=below:{$v_{n-4}$}] (c') at (13.2,0) {};
\node[another, label=below right:{$v_{n-3}$}] (d') at (14.1,0) {};
\node[another, label=above:{$v_{n-2}$}] (e') at (13.7,1) {};
\foreach \to/\from in {u1/u2, u2/u3, u1/a', u3/b', u2/a', u3/u4, u2/b', c'/d', c'/e', d'/e', n5/c'}
\draw [-] (\to)--(\from);
\node[label=above:{$\cdots \cdots$}](1) at (11.6,-.38){};
\end{tikzpicture}
\caption{The Saw graphs $S(3,0;n-7)$ (left) and $S(2,1;n-7)$ (right).}
\label{saw_picture}
\end{figure}

The {\it Broom tree} $B_{\Delta}^{(n)}$ is the tree obtained from a path graph $P_{n-\Delta +1}$ by attaching $\Delta -1$ pendant edges at an arbitrary leaf vertex of the path $P_{n-\Delta +1}$, see Fig~\ref{broom_picture}. In \cite{stevan}, it was shown that for every $\Delta > 2$,
\begin{equation}\label{broom_chain}
\rho (B_{\Delta}^{(n)})<\rho (B_{\Delta -1}^{(n)}).
\end{equation}

\begin{theorem}\cite{stevan}\label{broom_thm}
Let $T$ be a tree of order $n$ with $\Delta (T)=\Delta$ and $T\ncong B_{\Delta}^{(n)}$.  Then, $$\rho (T)< \rho (B_{\Delta}^{(n)}).$$ 
\end{theorem}

\begin{figure}[h!]
\begin{center}
\begin{tikzpicture}
[scale=1, vertices/.style={draw, fill=black, circle, minimum size = 10pt, inner sep=2pt}, another/.style={draw, fill=black, circle, minimum size = 3.5pt, inner sep=0.1pt}]
\node[another, label=below:{$v_1$}] (v1) at (0,-.5) {};
\node[another, label=below right:{$v_3$}] (v3) at (1,0) {};
\node[another, label=below right:{$v_4$}] (v4) at (2,0) {};
\node[another, label=above:{$v_{n-1}$}] (a) at (0,.5) {};
\node[another, label=right:{$v_{n}$}] (b) at (1,1) {};
\node[another, label=below:{$v_{n-3}$}] (c) at (3.6,0) {};
\node[another, label=below right:{$v_{n-2}$}] (d) at (4.6,0) {};
\node[another, label=right:{$v_2$}] (v2) at (1,-1) {};
\foreach \to/\from in {v2/v3, v3/b, v3/a, v1/v3, v3/v4, c/d}
\draw [-] (\to)--(\from);
\node[label=above:{$\, \cdots \cdots$}](1) at (2.8,-.36){};
\end{tikzpicture}
\caption{The broom tree $B_{5}^{(n)}$.}
\label{broom_picture}
\end{center}
\end{figure}
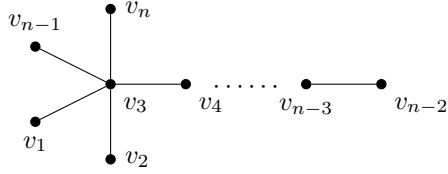

The {\it Perron vector} of a graph $G$ is the unique normalized eigenvector with positive coordinates which corresponds to the eigenvalue $\rho (G)$. In \cite{chromDist}, it was shown that  the coordinates of the Perron vector of a graph corresponding to its twin vertices are equal.

\section{Reduction to critical graphs with attached paths}

\begin{lemma}\label{k_critical_lemma}
If $G$ is a graph maximizing the distance spectral radius in $\mathcal{G}_{k,n}$, then $G$ has a $k$-critical subgraph $H$ and $G$ is obtained from $H$ by attaching paths at vertices of $H$. Moreover, if $S$ is the subset of vertices of $H$ at which nontrivial paths are attached, then $S$ is an independent set of $H$.
\end{lemma}

\begin{proof}
It is well known that every $k$-chromatic graph contains a $k$-critical subgraph. Let $H$ be a $k$-critical subgraph of an extremal graph $G$. If $G$ contains a cycle $C$ having an edge $e$ such that $e\in E(C)\setminus E(H)$, then $G\setminus e$ is a connected, $k$-chromatic graph on $n$ vertices. By Theorem~\ref{subgraph}, we get $\rho (G\setminus e)>\rho (G)$ but the latter contradicts with $G$ being an extremal graph in $\mathcal{G}_{k,n}$. Thus, an extremal graph $G$ must be obtained from a $k$-critical graph $H$ by attaching trees at some vertices of $H$. Let $T$ be a tree attached to $H$ at a vertex $v$. We shall show that $T$ must be a path graph and $v$ is a leaf vertex of $T$. Suppose on the contrary that either $T$ is  not a path graph or $deg_T(v)\geq 2$. Now, there exist a vertex $u$  of $T$ such that $T\setminus u$ has at least two components which are paths, say $P_k$ and $P_l$ where $k\geq l\geq 1$. Let $G'=G\setminus \left(V(P_k)\cup V(P_l)\right)$, then $G=G'(u,k,l)$. By Lemma~\ref{stevan}, we have $\rho (G'(u,k+1,l-1))> \rho (G'(u,k,l))$. Note that $G'(u,k+1,l-1)$ is a graph in $\mathcal{G}_{k,n}$ and the latter inequality contradicts with $G$ being an extremal graph. 

If $H$ has two adjacent vertices $x$ and $y$ such that two nontrivial paths are attached to $H$ at each of $x$ and $y$, then we can apply Lemma~\ref{godsil} to obtain another graph in $\mathcal{G}_{k,n}$ which has larger distance spectral radius than $G$. Thus, if $S$ is the subset of vertices of $H$ at which nontrivial paths are attached, then all vertices in $S$ are non-adjacent to each other, that is, $S$ is an independent set of $H$.
\end{proof}

\begin{corollary}\label{3chromatic}
	The kite graph $K_{3}^{(n)}$ is the unique graph maximizing the distance spectral radius in the family $\mathcal{G}_{3,n}$.
\end{corollary}

\begin{proof}
Let $G$ be an extremal graph in $\mathcal{G}_{3,n}$. It is well known that $3$-critical graphs are precisely the odd cycles. By Lemma~\ref{k_critical_lemma}, the graph $G$ belongs to $\mathcal{C}(n,1)$. By Theorem~\ref{cacti}, $\rho (S(1,0;n-3)) \geq \rho (G)$ with equality if and only if $G\cong S(1,0;n-3)$. Now the result follows since $S(1,0;n-3)\cong K_3^{(n)}$.
\end{proof}

\section{$4$-chromatic graphs $G$ with $\Delta (G)\geq 5$}

\begin{lemma}\label{broom_kite_lemma}
Let $n$ be an integer with $n\geq 6$, then $\rho(B_{5}^{(n)})<\rho (K_4^{(n)})$.
\end{lemma}
\begin{proof}
If $n\leq 12$, then the result follows from direct calculations in Table~\ref{table_saw_kite}. We may assume that $n\geq 13$. We consider the vertex labellings of $K_4^{(n)}$ and $B_{5}^{(n)}$ in Figures~\ref{kite_picture} and \ref{broom_picture} respectively.   Let $D=D(K_4^{(n)})$, $\rho(K_4^{(n)})=\rho$ and $D'=D(B_{5}^{(n)})$, $\rho(B_5^{(n)})=\rho '$. For the Perron vector $x$ of $B_5^{(n)}$, we have $$\frac{1}{2}x^T(D-D')x=-x_n(x_1+x_2+x_{n-1})-x_{n-1}(x_1+x_2)-x_1x_2+(x_n+x_{n-1}+x_1)\sum_{i=3}^{n-2}x_i.$$
Note that $x_1=x_2=x_{n-1}=x_n$, since  $v_1, v_2, v_{n-1}, v_n$ are twin vertices of $B_{5}^{(n)}$. So,
$$\frac{1}{2}x^T(D-D')x=-6x_1^2+3x_1\sum_{i=3}^{n-2}x_i=3x_1\left(-2x_1+\sum_{i=3}^{n-2}x_i\right).$$
By the eigenequation $D'x=\rho 'x$, we obtain
$\displaystyle \rho ' x_1=6x_1+\sum_{j=3}^{n-2}(j-2)x_j$ and
$$\rho ' x_i=4(i-2)x_1+\sum_{j=3}^{n-2}|i-j|\,x_j \quad  \text{for} \quad 3\leq i\leq n-2.$$
Now we find that 
\begin{eqnarray*}
\rho ' \left[-2x_1+\sum_{i=3}^{n-2}x_i \right] &=& \left[ -12+4\sum_{i=3}^{n-2}(i-2)\right]x_1+\sum_{j=3}^{n-2}\left[-2(j-2)+\sum_{i=3}^{n-2}|i-j| \right]x_j\\
&=&2(n-1)(n-6)x_1\, +\, \sum_{j=3}^{n-2}\left[\frac{1}{2}n^2-\left(\frac{3}{2}+j\right)n-3j+8+j^2 \right]x_j.
\end{eqnarray*}
It suffices to check that $p_j(n)=\frac{1}{2}n^2-\left(\frac{3}{2}+j\right)n-3j+8+j^2>0$ for every $n\geq 13$ and $3\leq j\leq n-2$. The roots of the quadratic $p_j(n)$ are 
$$\frac{3}{2}+j+\frac{1}{2}\sqrt{-55+36j-4j^2}\quad \text{and} \quad \frac{3}{2}+j-\frac{1}{2}\sqrt{-55+36j-4j^2}.$$
Observe that $p_j(n)$ has non-real roots when $j\geq 7.05$. Since $p_j(n)$ has positive leading coefficient, we obtain that $p_j(n)>0$ for all $n\geq 13$ when $j\geq 8$. Now we only need to consider the range $3\leq j\leq 7$. The maximum value of $f(j)=-55+36j-4j^2$ is achieved at the unique critical point of $j=9/2$ and it is equal to  $26$. So, the largest root of $p_j(n)$ is at most $\frac{3}{2}+j+\frac{\sqrt{26}}{2}$. It is easy to see that $\frac{3}{2}+j+\frac{\sqrt{26}}{2}<n$  because $3\leq j\leq 7$ and $n\geq 13$. It follows that $p_j(n)>0$ for all $n\geq 13$ and $3\leq j\leq n-2$. Thus, $\rho - \rho ' \geq x^T(D-D')x>0$.
\end{proof}

\begin{table}[H]
\begin{tabular}{ |c|c|c|c|c|c| }
\hline
$n$ & $\rho (S(3,0;n-7))$ & $\rho (S(2,1;n-7))$ & $\rho (B_5^{(n)})$ &$\rho (K_4^{(n)})$ \\
\hline
\hline
$6$ &          &          & $8.582$ &  $8.627$      \\
$7$ & $10.830$ & $10.830$ & $11.828$ & $12.727$ \\ 
$8$ & $14.462$ & $15.404$ & $16.090$ & $17.599$ \\ 
$9$ & $19.177$ & $20.784$ & $21.238$ & $23.219$ \\ 
$10$ &  $24.808$ & $26.940$ & $27.206$ & $29.575$ \\ 
$11$ & $31.279$ & $33.850$ & $33.959$ & $36.657$ \\ 
$12$ & $38.550$ & $41.503$ & $41.475$ & $44.460$ \\ 
\hline
\end{tabular}
\caption{Maximal distance spectral radii of certain saw, broom  and kite graphs limited to three decimals.}
\label{table_saw_kite}
\end{table}

\begin{lemma}\label{Delta5}
If $G$ is a connected graph of order $n$ with $\Delta (G)=\Delta \geq 5$, then $\rho (G) < \rho (K_4^{(n)})$.
\end{lemma}
\begin{proof}
Let $u$ be a vertex of maximum degree and $v_1,\dots , v_{\Delta}$ be the neighbors of $u$ in $G$. A minimal connected spanning subgraph of $G$ containing all edges $uv_i$ for $i\in \{1,\dots , \Delta \}$ is a spanning tree of $G$ with maximum degree $\Delta$. For such spanning tree $T$, we have $\rho (G)\leq \rho (T)$ by Theorem~\ref{subgraph} and $\rho (T)\leq \rho (B_{\Delta}^{(n)})$ by Theorem~\ref{broom_thm}. Also, $\rho (B_{\Delta}^{(n)})\leq \rho (B_5^{(n)})$ by the inequality in \eqref{broom_chain} and $\rho (B_5^{(n)}) < \rho (K_4^{(n)})$ by Lemma~\ref{broom_kite_lemma}. Thus, we obtain that $\rho (G) \leq \rho (T) \leq \rho (B_{\Delta}^{(n)})\leq \rho (B_5^{(n)}) < \rho (K_4^{(n)})$ and the result follows.
\end{proof}

\section{$4$-chromatic graphs with three cactus-type cycles}

\begin{lemma}\label{saw_3_0} For every $n\geq 7$, $\rho (S(3,0;n-7)) < \rho (K_4^{(n)}) $.
\end{lemma}
\begin{proof}
If $7\leq n\leq 10$, then the result follows from direct calculations in Table~\ref{table_saw_kite}. We may suppose that $n\geq 11$.
Let $D$ and $D'$ be the distance matrices of $K_4^{(n)}$ and $S(3,0;n-7)$ with respect to vertex orderings given in Figures~\ref{kite_picture} and \ref{saw_picture} respectively. Let also $x=(x_1, \dots , x_n)$ be the Perron vector of $D'$ and $\rho '=\rho (S(3,0;n-7))$. Note that
\[ D-D'=\begin{bmatrix}
    0 & 0 & 0 & 0 & 1 & \dots  & 1 & 0 & -1 \\
    0 & 0 & 0 & 0 & 1 & \dots  & 1 & 0 & 0 \\
    0 & 0 & 0 & 0 & 1 & \dots  & 1 & 0 & 1 \\
    0 & 0 & 0 & 0 & 0 & \dots  & 0 & 0 & 1 \\
    1 & 1 & 1 & 0 & 0 & \dots  & 0 & 1 & 2 \\
    \vdots & \vdots & \vdots & \vdots & \vdots & \ddots & \vdots & \vdots & \vdots \\
    1 & 1 & 1 & 0 & 0 & \dots & 0 & 1 & 2 \\
    0 & 0 & 0 & 0 & 1 & \dots & 1 & 0 & -1 \\
    -1 & 0 & 1 & 1 & 2 & \dots & 2 & -1 & 0
\end{bmatrix}.
\]
The vertices $v_1$ and $v_{n-1}$ of $S(3,0;n-7)$ are twin vertices, so $x_1=x_{n-1}$. Now we find that 
$$\frac{1}{2}x^T(D-D')x=(2x_1+x_2+x_3)\left(\sum_{j=5}^{n-2}x_j\right)+2x_n\left(-x_1+\sum_{j=5}^{n-2}x_j\right)+x_n(x_3+x_4).$$
 It suffices to check that $-x_1+\sum_{j=5}^{n-2}x_j>0$, as the Perron vector $x$ has positive coordinates. By the eigenequation $D'x=\rho ' x$, we get
 \begin{eqnarray*}
 \rho ' x_1 &=& x_1+x_2+2x_3+3x_4+2x_n+\sum_{k=5}^{n-2}(k-2)x_k\\
 \rho ' x_j &=& 2(j-2)x_1+(j-3)(x_2+x_n)+(j-4)(x_3+x_4)+\sum_{k=5}^{n-2}|k-j|x_k
 \end{eqnarray*}
 where $5\leq j\leq n-2$. Let us write $\displaystyle \rho ' \left(-x_1+\sum_{j=5}^{n-2}x_j\right)=\sum_{k\in\{1,\dots ,n-2,n\}}f_k(n)\,x_k$. Note that 
$f_1(n)=-1+2\sum_{j=5}^{n-2}(j-2)$, $f_2(n)=-1+\sum_{j=5}^{n-2}(j-3)$, $f_3(n)=-2+\sum_{j=5}^{n-2}(j-4)$, $f_4(n)=-1+f_3(n)$, $f_n(n)=-1+f_2(n)$. It is clear that $f_k(n)>0$ for $k\in\{1,2,3,4,n\}$ as $n\geq 11$. Also, for $5\leq k\leq n-2$, 
$$f_k(n)=-(k-2)+\sum_{j=5}^{n-2}|k-j|=\frac{1}{2}n^2-\left(\frac{3}{2}+k\right)n+k^2-4k+13.$$
Also, $f_k(n)$ is a quadratic polynomial in $n$ with a positive leading coefficient and the roots of $f_k(n)$ are 
$$r_1= \frac{3}{2}+k+\frac{1}{2}\sqrt{-4k^2+44k-95} \quad \quad \text{and} \quad \quad r_2=\frac{3}{2}+k-\frac{1}{2}\sqrt{-4k^2+44k-95}.$$
 Observe that $r_1$ and $r_2$ are non-real numbers unless $2.95\leq k\leq 8.05$, so $f_k(n)>0$ for all $n$ when $k\geq 9$. Now, we may suppose that $5\leq k\leq 8$. In this case it is easy to check that $n>r_1$ and $n>r_2$ because $n\geq 11$ by the assumption. Therefore, we get $f_k(n)>0$ for all $n\geq 11$ when $5\leq k\leq 8$.
\end{proof}

\begin{lemma}\label{saw_2_1} For every $n\geq 7$, $\rho (S(2,1;n-7) < \rho (K_4^{(n)})$.
\end{lemma}
\begin{proof}
If $n\leq 12$, then the result follows from direct calculations in Table~\ref{table_saw_kite}. We may assume that $n\geq 13$.
Let $D$ and $D'$ be the distance matrices of $K_4^{(n)}$ and $S(2,1;n-7)$ with respect to vertex orderings given in Figures~\ref{kite_picture} and \ref{saw_picture} respectively. Let also $x=(x_1, \dots , x_n)$ be the Perron vector of $D'$ and $\rho '=\rho (S(2,1;n-7))$. Note that

\[ D-D'=\begin{bmatrix}
    0 & 0 & 0 & \dots  & 0 & 0 & 1 & 0 & -1 \\
    0 & 0 & 0 & \dots  & 0 & 0 & 1 & 0 & 0 \\
    0 & 0 & 0 & \dots  & 0 & 0 & 1 & 0 & 1 \\
    \vdots & \vdots & \vdots & \ddots & \vdots & \vdots & \vdots & \vdots & \vdots \\
    0 & 0 & 0 & \dots & 0 & 0 & 1 & 0 & 1\\
    0 & 0 & 0 & \dots & 0 & 0 & 0 & 0 & 1\\
    1 & 1 & 1 & \dots & 1 & 0 & 0 & 1 & 2\\
    0 & 0 & 0 & \dots & 0 & 0 & 0 & 0 & -1\\
    -1 & 0 & 1 & \dots & 1 & 1 & 2 & -1 & 0\\
\end{bmatrix}.
\]
 Observe that the pairs $v_1$, $v_{n-1}$ and $v_{n-2}$, $v_{n-3}$ of $S(2,1;n-7)$ are twin vertices, so $x_1=x_{n-1}$ and $x_{n-2}=x_{n-3}$. Now, 
$$\frac{1}{2}x^T(D-D')x=x_{n-2}\left( x_{n-1}+\sum_{j=1}^{n-4}x_j\right)+x_n\left(3x_{n-2}-2x_1+\sum_{j=3}^{n-4}x_j\right).$$
It suffices to only show that $\displaystyle 3x_{n-2}-2x_1+\sum_{j=3}^{n-4}x_j>0$.  By the eigenequation $D'x=\rho ' x$, 
\begin{eqnarray*}
\rho ' x_{n-2} &=& x_{n-3}+(n-4)x_{n-1}+(n-5)x_n +\sum_{k=1}^{n-4}(n-3-k)x_k \\
\rho ' x_1 &=& (n-4)x_{n-2}+x_{n-1}+2x_n +\sum_{k=2}^{n-3}(k-1)x_k \\
\rho ' x_j &=& (n-3-j)x_{n-2}+(j-1)x_{n-1}+(j-2)x_n +\sum_{k=1}^{n-3}|j-k|\, x_k\\
\end{eqnarray*}
where $3\leq j\leq n-4$. Let us write 
$$\rho ' \left( 3x_{n-2}-2x_1+\sum_{j=3}^{n-4}x_j \right)=\sum_{k=1}^n f_k(n) \, x_k.$$
Since $\rho ' >0$ and all $x_k$'s positive, it suffices to show that $f_k(n) >0$ for every $k\in \{1, \dots , n\}$. It is straightforward to check that $f_k(n) >0$ when $k\in\{1, n-3, n-2, n-1, n\}$, as $n\geq 13$. Suppose that $2\leq k\leq n-4$, then
$$
f_k(n) = 3(n-3-k)-2(k-1)+\sum_{j=3}^{n-4}|k-j|
= \frac{1}{2}n^2-\left( \frac{1}{2}+k\right)n-4k+2+k^2.
$$
The roots of the quadratic $f_k(n)$ are 
$$r_1=\frac{1}{2}+k+\frac{1}{2}\sqrt{-15+36k-4k^2} \quad \text{and}\quad r_2=\frac{1}{2}+k-\frac{1}{2}\sqrt{-15+36k-4k^2}. $$

Observe that $f_k(n)$ has no real roots when $k\geq 9$ because $-15+36k-4k^2<0$ for $k\geq 9$. Also,  when $2\leq k\leq 8$, it is straightforward to verify that $n>r_1$ and $n>r_2$, as $n\geq 13$. Thus, $f_k(n)>0$ when $n\geq 13$ and $k\geq 2$, and the result follows.
\end{proof}

\begin{lemma}\label{cactus_not_extr} If $G$ is a connected graph of order $n$ having three cactus-type cycles, then $\rho(G)< \rho (K_4^{(n)})$.
\end{lemma}

\begin{proof}
Let $H$ a minimal spanning subgraph of $G$ containing three cactus-type cycles in $G$. So, $H$ is a cactus subgraph of $G$ which belongs to $\mathcal{C}(n,3)$.  By Theorem~\ref{subgraph}, we have $\rho (H)\geq \rho (G)$. By Theorem~\ref{cacti}, we have $\rho(H)\leq \rho (S(2,1;n-7)$ and $\rho (H)\leq \rho (S(3,0;n-7))$. Also, $\rho (K_4^{(n)})$ is larger than each of $\rho (S(3,0;n-7))$ and $\rho (S(2,1;n-7))$ by Lemmas~\ref{saw_3_0} and \ref{saw_2_1}. Thus, $\rho (G)\leq \rho (H)<\rho (K_4^{(n)})$ and the result follows.
\end{proof}

We say that a graph $G$ satisfies the {\it property} $\mathcal{P}$ if either $\Delta (G)\geq 5$ or $G$ contains three cactus-type cycles. By Lemmas~\ref{Delta5} and \ref{cactus_not_extr} we have the following:

\begin{corollary}\label{property_P} If $G$ is a graph of order $n$ satisfying the property $\mathcal{P}$, then $\rho(G)<\rho (K_4^{(n)})$
\end{corollary}

\section{$4$-critical planar graphs with exactly four triangles}

The well known Gr{\"u}nbaum-Aksenov Theorem \cite{aksenov,grunbaum} says that every $4$-chromatic planar graph contains at least four triangles. In \cite{lidicky}, a characterization of planar $4$-critical graphs with exactly four triangles was given. In this section we will give this characterization and we follow the definitions given in \cite{lidicky}. Let diamond and tailed diamond graphs be as shown in  Fig.~\ref{diamonds}. We say that an edge $e$ is a {\it diamond edge} of a graph $G$ if $e$ belongs to exactly two triangles in $G$. A {\it diamond expansion} of a graph $G$ is defined as follows: delete some diamond edge $xy$ from $G$, and identify $x$ with a leaf vertex of the tailed diamond,  and identify $y$ with a degree $2$ vertex of the tailed diamond graph.

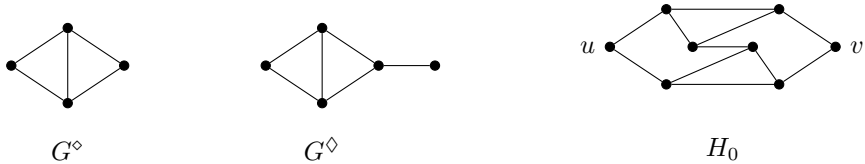
\begin{figure}[h!]
\begin{center}
\begin{tikzpicture}
[scale=1, vertices/.style={draw, fill=black, circle, minimum size = 10pt, inner sep=2pt}, another/.style={draw, fill=black, circle, minimum size = 3.5pt, inner sep=0.1pt}]
\node[another, label=right:{}] (u1) at (0,.5) {};
\node[another, label=right:{}] (u2) at (0,-.5) {};
\node[another, label=right:{}] (v1) at (-.75,0) {};
\node[another, label=right:{}] (v2) at (.75,0) {};
\foreach \to/\from in {u1/v2, u1/v1, u2/v2, u2/v1, u1/u2}
\draw [-] (\to)--(\from);
\node[label=above:{$G^{\diamond}$}](G1) at (0,-1.5){};
\end{tikzpicture}
\hspace{.5in}
\begin{tikzpicture}
[scale=1, vertices/.style={draw, fill=black, circle, minimum size = 10pt, inner sep=2pt}, another/.style={draw, fill=black, circle, minimum size = 3.5pt, inner sep=0.1pt}]
\node[another, label=right:{}] (u1) at (0,.5) {};
\node[another, label=right:{}] (u2) at (0,-.5) {};
\node[another, label=right:{}] (v1) at (-.75,0) {};
\node[another, label=right:{}] (v2) at (.75,0) {};
\node[another, label=right:{}] (t) at (1.5,0) {};
\foreach \to/\from in {u1/v2, u1/v1, u2/v2, u2/v1, u1/u2,v2/t}
\draw [-] (\to)--(\from);
\node[label=above:{$G^{\Diamond}$}](G2) at (0,-1.5){};
\end{tikzpicture}
\hspace{.5in}
\begin{tikzpicture}
[scale=1, vertices/.style={draw, fill=black, circle, minimum size = 10pt, inner sep=2pt}, another/.style={draw, fill=black, circle, minimum size = 3.5pt, inner sep=0.1pt}]
\node[another, label=left:{$u$}] (a) at (-1.5,0) {};
\node[another, label=right:{}] (b) at (-.75,-.5) {};
\node[another, label=right:{}] (c) at (.75,-.5) {};
\node[another, label=right:{}] (d) at (.4,0) {};
\node[another, label=right:{}] (e) at (-.4,0) {};
\node[another, label=right:{}] (f) at (-.75,.5) {};
\node[another, label=right:{}] (g) at (.75,.5) {};
\node[another, label=right:{$v$}] (h) at (1.5,0) {};
\foreach \to/\from in {a/b,a/f, b/c, f/g, b/d, d/c, e/f, e/g, d/e, c/h, h/g}
\draw [-] (\to)--(\from);
\node[label=above:{$H_0$}](G3) at (0,-1.75){};
\end{tikzpicture}
\caption{Diamond $G^{\diamond}$, tailed diamond $G^{\Diamond}$ and Havel's quasi-edge $H_0$.}
\label{diamonds}
\end{center}
\end{figure}

Let $\mathcal{TW}$ denote the family of graphs that can be obtained from $K_4$ by diamond expansions. Observe that every edge of $K_4$ is a diamond edge and a diamond expansion of $K_4$ using any edge of $K_4$ is isomorphic to the Moser spindle graph $M$ shown in Fig.\ref{diam_exp_1}. The Moser spindle has two diamond edges, namely $aa'$ and $bb'$, and end-vertices of each diamond edge are twin vertices. Using the symmetry of the graph, it is easy to see that any diamond expansion of $M$ is isomorphic to the graph $T$ in Fig.~\ref{diam_exp_1}. Note that $T$ has two diamond edges, namely, $\alpha \alpha '$ and $\beta \beta '$. Also, the graph $T$ contains a disjoint union $2K_3\cup C_4$ where each triangle $K_3$ 
in the disjoint union contains a diamond edge of $T$ (the thick edges of $T$ in Fig.~\ref{diam_exp_1} form a $2K_3\cup C_4$). Therefore, all diamond expansions of $T$ must contain a $2K_3\cup C_4$ as well.

\begin{figure}[h!]
\begin{center}
\begin{tikzpicture}
[scale=1, vertices/.style={draw, fill=black, circle, minimum size = 10pt, inner sep=2pt}, another/.style={draw, fill=black, circle, minimum size = 3.5pt, inner sep=0.1pt}]
\node[another, label=right:{}] (a) at (0,1.25) {};
\node[another, label=right:{}] (b) at (-.5,0) {};
\node[another, label=right:{}] (c) at (.5,0) {};
\node[another, label=right:{}] (d) at (0,-1.25) {};
\foreach \to/\from in {a/b, a/c, b/c, b/d, c/d}
\draw [-] (\to)--(\from);
\draw (d) to [out=170,in=190,looseness=1.5] (a);
\node[label=above:{$K_4$}](G1) at (0,-2){};
\end{tikzpicture}
\hspace{.1in}
\begin{tikzpicture}
[scale=1, vertices/.style={draw, fill=black, circle, minimum size = 10pt, inner sep=2pt}, another/.style={draw, fill=black, circle, minimum size = 3.5pt, inner sep=0.1pt}]
\node[another, label=left:{$e$}] (e) at (0,1.25) {};
\node[another, label=right:{$a$}] (a) at (-.75,0) {};
\node[another, label=above:{$a'$}] (a') at (-1.5,.25) {};
\node[another, label=left:{$b$}] (b) at (.75,0) {};
\node[another, label=above:{$b'$}] (b') at (1.5,.25) {};
\node[another, label=left:{$d$}] (d) at (-.75,-1.25) {};
\node[another, label=right:{$c$}] (c) at (.75,-1.25) {};
\foreach \to/\from in {e/a, e/a', e/b, e/b', a/a', b/b', a/d, a'/d, b/c, b'/c, c/d}
\draw [-] (\to)--(\from);
\node[label=above:{$M$}](M) at (0,-2){};
\end{tikzpicture}
\hspace{.1in}
\begin{tikzpicture}
[scale=1, vertices/.style={draw, fill=black, circle, minimum size = 10pt, inner sep=2pt}, another/.style={draw, fill=black, circle, minimum size = 3.5pt, inner sep=0.1pt}, another2/.style={draw, circle, minimum size = 3.5pt, inner sep=0.1pt}]
\node[another2, label=left:{}] (e) at (0,1.25) {};
\node[another2, label=right:{}] (a) at (-.25,0) {};
\node[another2, label=above:{}] (a') at (-2.25,0) {};
\node[another, label=left:{$\beta$}] (b) at (.75,0) {};
\node[another, label=right:{$\beta '$}] (b') at (1.5,.25) {};
\node[another2, label=left:{}] (d) at (-.75,-1.25) {};
\node[another2, label=right:{}] (c) at (.75,-1.25) {};
\node[another, label=above:{$\alpha$}] (x) at (-1,.25) {};
\node[another, label=below:{$\alpha '$}] (y) at (-1,-.25) {};
\node[another2, label=right:{}] (z) at (-1.75,0) {};
\foreach \to/\from in {e/b, e/b', c/d, a/x, a/y,z/a'}
\draw [-] (\to)--(\from);
\draw [-, line width=1.5pt] (a') to [out=60,in=190,looseness=1] (e); 
\draw [-, line width=1.5pt] (d) to [out=160,in=270,looseness=1] (a'); 
\foreach \to/\from in {e/a, d/a, x/y, x/z, y/z, b/b', b'/c, b/c}
\draw [-, line width=1.5pt] (\to)--(\from);
\node[label=above:{$T$}](T) at (0,-2){};
\end{tikzpicture}
\caption{$K_4$, the Moser spindle $M$ and the graph $T$.}
\label{diam_exp_1}
\end{center}
\end{figure}
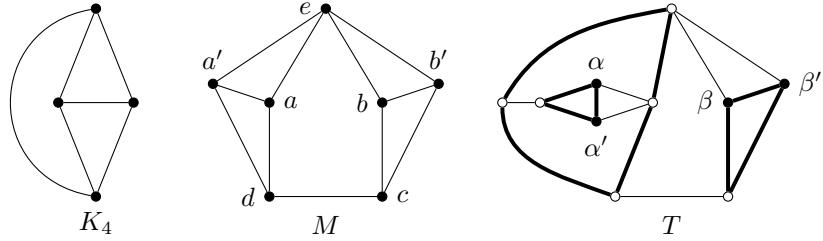

{\it Havel's quasi-edge} is the graph denoted by $H_0$ in Fig.~\ref{diamonds}. Let $u$ and $v$ be two vertices of degree $2$ in $H_0$. A {\it Havel's quasi-edge $H_0$ expansion} of a graph $G$ is defined as follows: delete some diamond edge $xy$ from $G$, and identify $x$ with the vertex $u$ of $H_0$ and identify $y$ with the vertex $v$ of $H_0$. Let $\mathcal{TW}_1$ denote the family of graphs that can be obtained from a graph in $\mathcal{TW}$ by a Havel's quasi-edge $H_0$ expansion.  Let $\mathcal{TW}_2$ denote the family of graphs that can be obtained from a graph in $\mathcal{TW}_1$ by a Havel's quasi-edge $H_0$ expansion. It is clear that every graph in $\mathcal{TW} _1\cup \mathcal{TW} _2$ contains a $3K_3$ because every graph in $\mathcal{TW}$ has two non-adjacent diamond edges and $H_0$ contains a $2K_3$ which contains neither of the vertices $u$ and $v$.

\begin{figure}[h!]
\begin{center}
\begin{tikzpicture}
[scale=1.25, vertices/.style={draw, fill=black, circle, minimum size = 10pt, inner sep=2pt}, another/.style={draw, fill=black, circle, minimum size = 3.5pt, inner sep=0.1pt}]
\node[another, label=left:{$x$}] (x) at (-.5,.5) {};
\node[another, label=right:{$z'$}] (z') at (.5,.5) {};
\node[another, label=left:{$y'$}] (y') at (-1,0) {};
\node[another, label=right:{$y$}] (y) at (1,0) {};
\node[another, label=left:{$z$}] (z) at (-.5,-.5) {};
\node[another, label=right:{$x'$}] (x') at (.5,-.5) {};
\node[another, label=right:{$r$}] (u) at (0,0) {};
\foreach \to/\from in {x/z', x/y',z'/y, y'/z, z/x', y/x', u/z', u/x', u/y'}
\draw [-] (\to)--(\from);
\node[label=above:{$Q_1$}](Q1) at (0,.6){};
\end{tikzpicture}
\begin{tikzpicture}
[scale=1.25, vertices/.style={draw, fill=black, circle, minimum size = 10pt, inner sep=2pt}, another/.style={draw, fill=black, circle, minimum size = 3.5pt, inner sep=0.1pt}]
\node[another, label=left:{$x$}] (x) at (-.5,.5) {};
\node[another, label=right:{$z'$}] (z') at (.5,.5) {};
\node[another, label=left:{$y'$}] (y') at (-1,0) {};
\node[another, label=right:{$y$}] (y) at (1,0) {};
\node[another, label=left:{$z$}] (z) at (-.5,-.5) {};
\node[another, label=right:{$x'$}] (x') at (.5,-.5) {};
\node[another, label=above:{$r$}] (r) at (-.25,0) {};
\node[another, label=right:{$s$}] (s) at (.25,0) {};
\foreach \to/\from in {x/z', x/y',z'/y, y'/z, z/x', y/x', r/y', r/z', s/x', s/z'}
\draw [-] (\to)--(\from);
\draw [-] (y') to [out=340,in=250,looseness=1] (s); 
\node[label=above:{$Q_2$}](Q2) at (0,.6){};
\end{tikzpicture}
\begin{tikzpicture}
[scale=1.25, vertices/.style={draw, fill=black, circle, minimum size = 10pt, inner sep=2pt}, another/.style={draw, fill=black, circle, minimum size = 3.5pt, inner sep=0.1pt}]
\node[another, label=left:{$x$}] (x) at (-.5,.5) {};
\node[another, label=right:{$z'$}] (z') at (.5,.5) {};
\node[another, label=left:{$y'$}] (y') at (-1,0) {};
\node[another, label=right:{$y$}] (y) at (1,0) {};
\node[another, label=left:{$z$}] (z) at (-.5,-.5) {};
\node[another, label=right:{$x'$}] (x') at (.5,-.5) {};
\node[another, label=left:{}] (r) at (0,.25) {};
\node[another, label=right:{}] (s) at (.3,-.1) {};
\node[another, label=right:{}] (t) at (-.25,-.15) {};
\foreach \to/\from in {x/z', x/y',z'/y, y'/z, z/x', y/x', t/y', r/z', r/s, t/s,t/x',s/y,r/y'}
\draw [-] (\to)--(\from);
\node[label=above:{$Q_3$}](Q3) at (0,.6){};
\end{tikzpicture}
\caption{Several quadrangulations of the interior of $C_6$.}
\label{quadr}
\end{center}
\end{figure}
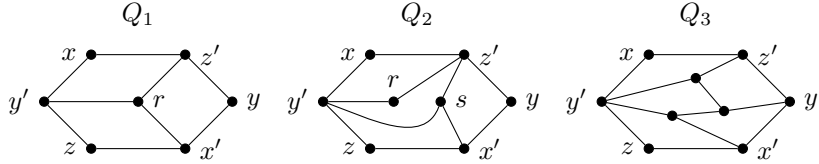

Let $G$ be a plane graph. We say that a subgraph $P$ is a {\it patch} in $G$ if $P$ is a quadrangulation of the interior of a $C_6$ shown by $xz'yx'zy'$, such that all neighbors of $x', y'$ and $z'$ are inside of $P$.  The patch $P$ is called {\it critical} if $x'$, $y'$ and $z'$ have at least three neighbors and every cycle of length four in $P$ bounds a face. Let $v$ be a vertex of $G$ with $N_G(v)=\{x,y,z\}$. We write $G_v$ for the graph obtained from $G \setminus v$ by inserting a critical patch $P$ whose boundary is a $C_6$ shown by $xz'yx'zy'$, with $x',y',z'$ being new vertices, and we say that $G_v$ is a {\it critical patch P expansion} of $G$ at the vertex $v$. We refer the reader to \cite{lidicky} for further details of the notions mentioned here. Critical patch expansions of $K_4$ using the quadrangulations $Q_1$ and $Q_2$ in Fig.~\ref{quadr} is shown in Fig.~\ref{myciel}. The graph $M^{'}$ in Fig.~\ref{myciel} is also known as the {\it Mycielskian of a triangle}.

\begin{figure}[h!]
\begin{center}
\begin{tikzpicture}
[scale=1.5, vertices/.style={draw, fill=black, circle, minimum size = 10pt, inner sep=2pt}, another/.style={draw, fill=black, circle, minimum size = 3.5pt, inner sep=0.1pt}]
\node[another, label=below:{$v$}] (v) at (0,0) {};
\node[another, label=below:{$z$}] (z) at (-.75,-.5) {};
\node[another, label=below:{$y$}] (y) at (.75,-.5) {};
\node[another, label=above:{$x$}] (x) at (0,.5) {};
\foreach \to/\from in {x/y, x/z, y/z, v/y, v/x, v/z}
\draw [-] (\to)--(\from);
\node[label=above:{$K_4$}](K4) at (0,-1.25){};
\end{tikzpicture}
\begin{tikzpicture}
[scale=1.5, vertices/.style={draw, fill=black, circle, minimum size = 10pt, inner sep=2pt}, another/.style={draw, fill=black, circle, minimum size = 3.5pt, inner sep=0.1pt}]
\node[another, label=above:{$x$}] (x) at (-.25,.5) {};
\node[another, label=right:{$z'$}] (z') at (.25,.5) {};
\node[another, label=left:{$y'$}] (y') at (-.45,0) {};
\node[another, label=right:{$y$}] (y) at (.75,0) {};
\node[another, label=below:{$z$}] (z) at (-.25,-.5) {};
\node[another, label=right:{$x'$}] (x') at (.25,-.5) {};
\node[another, label=right:{$r$}] (u) at (.1,0) {};
\foreach \to/\from in {x/z', x/y',z'/y, y'/z, z/x', y/x', u/z', u/x', u/y'}
\draw [-] (\to)--(\from);
\draw [-] (x) to [out=45,in=70,looseness=1.6] (y); 
\draw [-] (x) to [out=160,in=200,looseness=2] (z); 
\draw [-] (z) to [out=320,in=280,looseness=1.7] (y); 
\node[label=above:{$M^{'}$}](Mic) at (0,-1.3){};
\end{tikzpicture}
\begin{tikzpicture}
[scale=1.5, vertices/.style={draw, fill=black, circle, minimum size = 10pt, inner sep=2pt}, another/.style={draw, fill=black, circle, minimum size = 3.5pt, inner sep=0.1pt}]
\node[another, label=above:{$x$}] (x) at (-.25,.5) {};
\node[another, label=right:{$z'$}] (z') at (.25,.5) {};
\node[another, label=left:{$y'$}] (y') at (-.5,0) {};
\node[another, label=right:{$y$}] (y) at (.75,0) {};
\node[another, label=below:{$z$}] (z) at (-.25,-.5) {};
\node[another, label=right:{$x'$}] (x') at (.25,-.5) {};
\node[another, label=above:{$r$}] (u) at (-.05,0) {};
\node[another, label=right:{$s$}] (v) at (.3,0) {};
\foreach \to/\from in {x/z', x/y',z'/y, y'/z, z/x', y/x', u/y', u/z',v/z',v/x'}
\draw [-] (\to)--(\from);
\draw [-] (x) to [out=45,in=70,looseness=1.6] (y); 
\draw [-] (x) to [out=160,in=200,looseness=2] (z); 
\draw [-] (z) to [out=320,in=280,looseness=1.7] (y); 
\draw [-] (y') to [out=330,in=230,looseness=1.2] (v); 
\node[label=above:{$M^{''}$}](M2) at (0,-1.3){};
\end{tikzpicture}
\caption{Some critical patch expansions of $K_4$ at $v$.}
\label{myciel}
\end{center}
\end{figure}
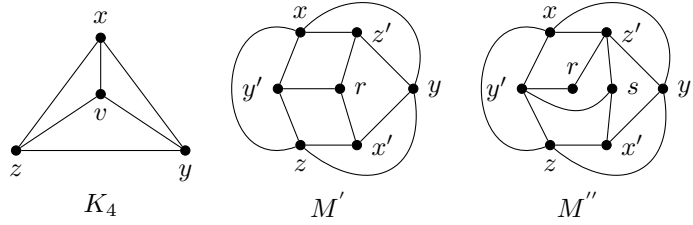

\begin{figure}[h!]
\begin{center}
\begin{tikzpicture}
[scale=1, vertices/.style={draw, fill=black, circle, minimum size = 10pt, inner sep=2pt}, another/.style={draw, fill=black, circle, minimum size = 3.5pt, inner sep=0.1pt}]
\node[another, label=left:{$a$}] (a) at (0,1.25) {};
\node[another, label=right:{$y$}] (y) at (-.75,0) {};
\node[another, label=above:{$x$}] (x) at (-1.5,.25) {};
\node[another, label=left:{$b$}] (b) at (.75,0) {};
\node[another, label=above:{$c$}] (c) at (1.5,.25) {};
\node[another, label=left:{$v$}] (v) at (-.75,-1.25) {};
\node[another, label=right:{$z$}] (z) at (.75,-1.25) {};
\foreach \to/\from in {a/y, a/x, a/b, a/c, y/x, b/c, y/v, x/v, b/z, c/z, z/v}
\draw [-] (\to)--(\from);
\node[label=above:{$M$}](M) at (0,-2.1){};
\end{tikzpicture}
\hspace{.2in}
\begin{tikzpicture}
[scale=1, vertices/.style={draw, fill=black, circle, minimum size = 10pt, inner sep=2pt}, another/.style={draw, fill=black, circle, minimum size = 3.5pt, inner sep=0.1pt}]
\node[another, label=right:{$a$}] (a) at (0,1.25) {};
\node[another, label=right:{$y$}] (y) at (-.5,0) {};
\node[another, label=right:{$x'$}] (x') at (-.25,-.75) {};
\node[another, label=left:{$x$}] (x) at (-2.5,0) {};
\node[another, label=left:{$y'$}] (y') at (-2.25,-.75) {};
\node[another, label=above:{$z'$}] (z') at (-1.75,-.25) {};
\node[another, label=below:{}] (u) at (-1.25,-.95) {};
\node[another, label=left:{$b$}] (b) at (.75,0) {};
\node[another, label=above:{$c$}] (c) at (1.5,.25) {};
\node[another, label=right:{$z$}] (z) at (.7,-1.3) {};
\foreach \to/\from in {a/y, a/b, a/c, b/c, b/z, c/z, x/z',z'/y,x/y',u/z',y/x'}
\draw [-] (\to)--(\from);
\draw [-] (x) to [out=70,in=180,looseness=1] (a); 
\draw [-, line width=1.5pt] (x) to [out=60,in=100,looseness=.8] (y); 
\draw[-,line width=1.5pt] (a) to (b);
\draw[-,line width=1.5pt] (a) to (c);
\draw[-,line width=1.5pt] (c) to (b);
\draw[-,line width=1.5pt] (z') to (y);
\draw[-,line width=1.5pt] (x) to (z');
\draw[-,line width=1.5pt] (u) to (y');
\draw[-,line width=1.5pt] (u) to (x');
\draw [-, line width=1.5pt] (z) to [out=160,in=280,looseness=1] (x'); 
\draw [-, line width=1.5pt] (z) to [out=180,in=320,looseness=1] (y'); 
\node[label=above:{$M_1$}](M1) at (0,-2.1){};
\end{tikzpicture}
\hspace{.2in}
\begin{tikzpicture}
[scale=1, vertices/.style={draw, fill=black, circle, minimum size = 10pt, inner sep=2pt}, another/.style={draw, fill=black, circle, minimum size = 3.5pt, inner sep=0.1pt}]
\node[another, label=right:{$a$}] (a) at (0,1.25) {};
\node[another, label=right:{$y$}] (y) at (-.5,0) {};
\node[another, label=right:{$x'$}] (x') at (-.2,-.65) {};
\node[another, label=left:{$x$}] (x) at (-2.5,0) {};
\node[another, label=left:{$y'$}] (y') at (-2.25,-.75) {};
\node[another, label=above:{$z'$}] (z') at (-1.75,-.25) {};
\node[another, label=below:{}] (u) at (-1,-.5) {};
\node[another, label=below:{}] (v) at (-1.25,-1) {};
\node[another, label=left:{$b$}] (b) at (.75,0) {};
\node[another, label=above:{$c$}] (c) at (1.5,.25) {};
\node[another, label=right:{$z$}] (z) at (.7,-1.3) {};
\foreach \to/\from in {a/y, b/z, c/z,x/y',y/x',u/x',u/z', v/z'}
\draw [-] (\to)--(\from);
\foreach \to/\from in {a/c,a/b,b/c,x/z',z'/y,y'/v,v/x'}
\draw [-, line width=1.5pt] (\to)--(\from);
\draw [-] (x) to [out=70,in=180,looseness=1] (a); 
\draw [-, line width=1.5pt] (x) to [out=60,in=100,looseness=.8] (y); 
\draw [-, line width=1.5pt] (z) to [out=180,in=310,looseness=1] (y'); 
\draw[-,line width=1.5pt] (x') to [out=300,in=170,looseness=1](z);
\node[label=above:{$M_2$}](M2) at (0,-2.1){};
\end{tikzpicture}
\caption{Some critical patch expansions of $M$ at $v$. The graphs $M_1$ and $M_2$ contain a $2K_3\cup C_4$ shown by bold edges.}
\label{moser_expansions}
\end{center}
\end{figure}
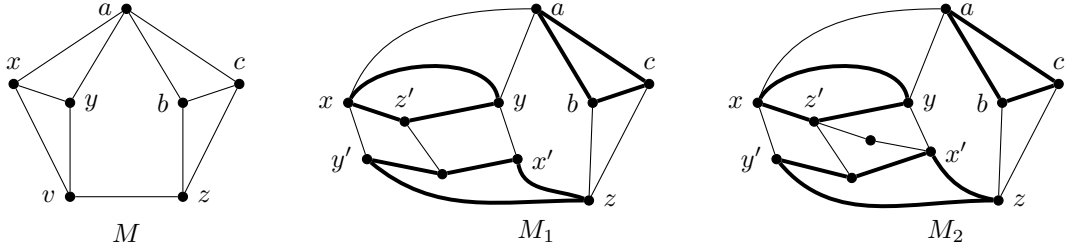

\begin{theorem} \cite{lidicky} A planar $4$-critical graph has exactly four triangles if and only if it is obtained from a graph in  $\mathcal{F}=\mathcal{TW}\cup \mathcal{TW}_1\cup \mathcal{TW}_2$ by replacing several (possibly zero) non-adjacent vertices of degree $3$ with critical patches.
\end{theorem}

\begin{figure}[h!]
\begin{center}
\begin{tikzpicture}
[scale=1.5, vertices/.style={draw, fill=black, circle, minimum size = 10pt, inner sep=2pt}, another/.style={draw, fill=black, circle, minimum size = 3.5pt, inner sep=0.1pt}]
\node[another, label=left:{$u_1$}] (u1) at (0,0) {};
\node[another, label=below:{$u_2$}] (u2) at (-1,-1) {};
\node[another, label=below:{$u_3$}] (u3) at (1,-1) {};
\node[another, label=below:{$v_1$}] (v1) at (0,-1) {};
\node[another, label=right:{$v_2$}] (v2) at (.5,-.5) {};
\node[another, label=left:{$v_3$}] (v3) at (-.5,-.5) {};
\foreach \to/\from in {u1/v2, u1/v3, v2/v3, v2/u3, v1/u3,v1/u2, u2/v3, v3/v1, v1/v2}
\draw [-] (\to)--(\from);
\end{tikzpicture}
\caption{The triangular grid graph $T^{*}$.}
\label{trian_grid_picture}
\end{center}
\end{figure}
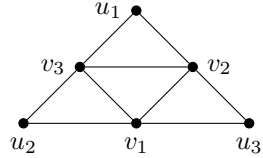

\begin{lemma}\label{trian_grid_lemma} Let $G$ be a $4$-critical planar graph. If $G$ contains a triangular grid graph, then either $G$ is the Mycielskian of a triangle $M^{'}$ or $G$ satisfies property $\mathcal{P}$.
\end{lemma}

\begin{proof} Let $T^{*}$ be a triangular grid in $G$ with vertices labelled as in Fig.~\ref{trian_grid_picture}. First, note that $V(G)\setminus V(T^{*})\neq \emptyset$ because the unique $4$-critical graph on six vertices is the wheel graph $K_1\vee C_5$ which clearly does not contain a $T^{*}$.  Let us suppose that $\Delta (G)=4$. If $G\setminus T^{*}$ has a cycle $C$ in it, then this cycle $C$ and the triangles $u_1v_2v_3$ and $v_1u_2v_3$ are three cactus-type cycles in $G$. So, for the rest, we may assume that $G\setminus T^{*}$ is a forest. We consider two cases.

\noindent {\it Case 1:} $G\setminus T^{*}$ has an isolated vertex $v$. Since $\Delta (G)=4$ and $\delta (G)\geq 3$, the vertex $v$ has no neighbor in $\{v_1,v_2,v_3\}$ and it must be adjacent to all of $u_1,u_2,u_3$. Now, $T^{*}$ together with $v$ yields the graph $M^{'}$ which is a $4$-critical graph. Note that a $4$-critical graph cannot contain a $4$-critical graph as a proper subgraph. Therefore, $G$ must indeed be isomorphic to $M^{'}$.

\noindent {\it Case 2:} $G\setminus T^{*}$ has no isolated vertices. Let $w_1$ and $w_2$ be two leaves of a connected component of $G\setminus T^{*}$. Since  $\Delta (G)=4$ and $\delta (G)\geq 3$, the vertices $w_1$ and $w_2$ have no neighbors in $\{v_1,v_2,v_3\}$ and have at least two neighbors in $\{u_1,u_2,u_3\}$. Hence, some vertex $u_i$ must be adjacent to both of $w_1$ and $w_2$. The edges $u_iw_1$ and $u_iw_2$ together with the unique path joining $v_1$ to $v_2$ in $G\setminus T^{*}$ form a cycle $C$. Now, the cycle $C$ together with the triangles $u_1v_2v_3$ and $v_1u_2v_3$ form three cactus-type cycles in $G$.
\end{proof}

\begin{corollary}\label{lidicky_corollary} Every planar $4$-critical graph with exactly four triangles which is different from $K_4$, the graph $M$ (the Moser spindle) and the graph $M^{'}$ (the Mycielskian of a triangle) satisfies the property $\mathcal{P}$.
\end{corollary}

\begin{proof} As we already observed earlier, every graph in $\mathcal{F}\setminus \{K_4, M\}$ contains three vertex disjoint cycles. It is easy to see that if a graph contains three vertex disjoint cycles, then so does every critical patch expansion of the graph. Therefore, we only need to consider critical patch expansions of $K_4$ and $M$. Let $G\ncong K_4$ be a critical patch expansion of $K_4$. Then, $G$ must contain a triangular grid $T^{*}$. Moreover, replacing degree $3$ vertices of a graph containing a triangular grid $T^{*}$ with critical patches yields a graph which also contains triangular grid. Now, by Lemma~\ref{trian_grid_lemma}, all graphs obtained from a $K_4$ by such critical patch expansions, except $K_4$ itself and the graph $M^{'}$, satisfy the property $\mathcal{P}$.

Suppose that $G$ is a graph obtained from $M$ by replacing an end-vertex of a diamond edge of $M$ by a critical patch. Let $u$ be the unique vertex of degree four in $M$. Observe that $deg_G(u)\geq 5$ because $u$ is adjacent to two new vertices in $G$ besides the three end-vertices of the two diamond edges. It is also clear that a critical patch expansion cannot decrease the maximum degree of the graph. Lastly, let us consider a critical patch expansion $G$ obtained from $M$ by replacing a degree-three vertex which does not belong to a diamond edge. It is straightforward to check that in this case $G$ has three vertex disjoint cycles, indeed $G$ contains a $2K_3\cup C_4$, and the result follows.
\end{proof}

\section{Planar $4$-critical graphs with at least five triangles}

\begin{lemma}\label{fan_lemma} Let  $G$ be a $4$-critical planar graph. If $G$ has at least five triangles  and $G$ contains the fan graph $K_1\vee P_4$, then $G$ satisfies $\mathcal{P}$.
\end{lemma}
\begin{proof}
Let $T_1,T_2,T_3$ be three triangles in a fan subgraph $K_1\vee P_4$ of $G$ and let $V(T_i)=\{u,v_i,v_{i+1}\}$ for $i\in\{1,2,3\}$. Note that $v_1v_3, v_2v_4\notin E(G)$ because otherwise $G$ would contain a $K_4$ which contradicts with $G$ being $4$-critical. We consider two cases.

\noindent {\it Case 1}: $v_1v_4\in E(G)$. Let $T$ be the triangle $uv_1v_4$ and let $T'$ be a triangle different from $T_1,T_2,T_3$ and $T$. If $T'$ does not share an edge with any of the triangles $T_1,T_2,T_3$ and $T$, then $T_1$, $T_3$ and $T'$ are three cactus-type triangles in $G$. If $T'$ shares an edge $uv_i$ for some $i$, then $\Delta (G)\geq deg_G(u)\geq 5$ and we are done. If $T'$ contains the edge $v_2v_3$, then the triangles $T_1,T_2,T_3$ and $T'$ form a triangular grid. By Lemma~\ref{trian_grid_lemma}, $G$ must satisfy property $\mathcal{P}$ because $G$ has at least five triangles and the graph $M^{'}$ has exactly four triangles. If $T'$ contains the edge $v_3v_4$, then the triangles $T,T',T_2,T_3$ form a triangular grid and we are done by Lemma~\ref{trian_grid_lemma} again. Similar argument works if $T'$ contains the edge $v_1v_2$ too. Lastly, suppose that $T'$ contains the edge $v_1v_4$. Now, the triangles $T,T',T_1,T_3$ form a triangular grid and the result follows from Lemma~\ref{trian_grid_lemma}.

\noindent {\it Case 2:} $v_1v_4\notin E(G)$. If there exists a triangle $T\notin\{T_1,T_2,T_3\}$ which does not share an edge with any of $T_1,T_2$ or $T_3$, then $T,T_1$ and $T_3$ are three cactus-type triangles. We may assume that every triangle $G$ shares an edge with some triangle $T_i$. If there exists a triangle $T\notin\{T_1,T_2,T_3\}$ and $T$ contains some edge $uv_i$ or $v_2v_3$, we proceed as in previous case.  Let $T$ and $T'$ be two distinct triangles different from $T_1,T_2,T_3$. We may assume that $T$ and $T'$ do not contain any of the edges $uv_i$ and $v_2v_3$ where $i\in \{1,2,3,4\}$. If both of $T$ and $T'$ contains the edge $v_1v_2$ or $v_3v_4$, then we get $deg_G(v_2)\geq 5$ or $deg_G(v_3)\geq 5$ respectively. Without loss of generality, suppose that $T$ contains the edge $v_1v_2$ and $T'$ contains the edge $v_3v_4$. Now, $T,T'$ and $T_2$ are three cactus-type triangles and we are done.

\end{proof}

\begin{figure}[h!]
\begin{center}
\begin{tikzpicture}
[scale=.75, vertices/.style={draw, fill=black, circle, minimum size = 10pt, inner sep=2pt}, another/.style={draw, fill=black, circle, minimum size = 3.5pt, inner sep=0.1pt}]
\node[another, label=right:{$u_1$}] (u1) at (0,1) {};
\node[another, label=right:{$u_2$}] (u2) at (0,-1) {};
\node[another, label=right:{$v_1$}] (v1) at (-1,0) {};
\node[another, label=right:{$v_2$}] (v2) at (1,0) {};
\node[another, label=left:{$v_3$}] (v3) at (-2,0) {};
\foreach \to/\from in {u1/u2, u1/v2, u1/v1, u1/v3, u2/v2, u2/v1, u2/v3}
\draw [-] (\to)--(\from);
\end{tikzpicture}
\hspace{.5in}
\begin{tikzpicture}
[scale=.75, vertices/.style={draw, fill=black, circle, minimum size = 10pt, inner sep=2pt}, another/.style={draw, fill=black, circle, minimum size = 3.5pt, inner sep=0.1pt}]
\node[another, label=right:{$u_1$}] (u1) at (0,1) {};
\node[another, label=right:{$u_2$}] (u2) at (0,-1) {};
\node[another, label=right:{$v_1$}] (v1) at (-.75,0) {};
\node[another, label=above:{$v_2$}] (v2) at (-3,0) {};
\node[another, label=left:{$v_3$}] (v3) at (-1.65,0) {};
\foreach \to/\from in {u1/u2, u1/v2, u1/v1, u1/v3, u2/v2, u2/v1, u2/v3}
\draw [-] (\to)--(\from);
\end{tikzpicture}
\caption{Two planar drawings of $K_2\vee \overbar{K}_3$.}
\label{3triangles_overlap_picture}
\end{center}
\end{figure}
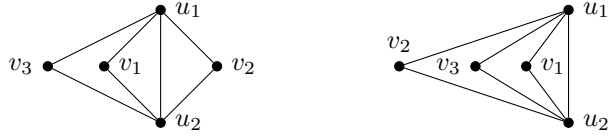

\begin{lemma}\label{3triangles_overlap} If $G$ is a $4$-critical planar graph containing $K_2\vee \overbar{K}_3$ as a subgraph, then $G$ satisfies property $\mathcal{P}$.
\end{lemma}

\begin{proof}
Let $u_1,u_2,v_1,v_2,v_3$ be the vertices of a $K_2\vee \overbar{K}_3$ subgraph such that $u_1u_2\in E(G)$ and $u_iv_j\in E(G)$ for $1\leq i\leq 2$ and $1\leq j\leq 3$. We may assume that $N_G(u_1)=\{u_2,v_1,v_2,v_3\}$ and $N_G(u_2)=\{u_1,v_1,v_2,v_3\}$ because otherwise we would have $deg_G(u_1)\geq 5$ or $deg_G(u_2)\geq 5$ both of which yields $\Delta (G)\geq 5$. Since $G$ is $4$-critical, $G$ is $K_4$-free. So, $v_iv_j\notin E(G)$ for every $i$ and $j$. For $i\in\{1,2\}$ we recursively define a sequence of  vertex subsets as follows:

\ \ $A_1^{(i)}=N_G(v_i)\setminus \{u_1,u_2\}$ and

\ \ $A_k^{(i)}=N_G[A_{k-1}^{(i)}]\setminus \{u_1,u_2,v_3\}$ for $k\geq 2$.

Let $l_i$ be the largest integer such that $A_{l_i}^{(i)}=A_{l_i+1}^{(i)}$. Let $H_i$ be the subgraph of $G$ induced by the vertex subset $A_{l_i}^{(i)}\cup \{v_3\}$ for $i\in \{1,2\}$. There are two different planar drawings of  $K_2\vee \overbar{K}_3$ (see Fig.~\ref{3triangles_overlap_picture}). The existence of a vertex $v$ such that $v\in V(H_1)\cap V(H_2)$ and $v\neq v_3$ violates the planarity of the graph. Hence, $V(H_1)\cap V(H_2) \subseteq \{v_3\}$. Note that $\delta (G)\geq 3$, as $G$ is $4$-critical. So the vertex $v_i$ has at least one neighbor in $V(H_i)\setminus \{v_3\}$ and every vertex $v \in V(H_i)\setminus \{v_i,v_3\}$ has at least three neighbors in $H_i$ for $i\in \{1,2\}$. Now, $H_i$ is a subgraph with at least three vertices and  $deg_{H_i}(v)\geq 3$ for every vertex $v$ in  $V(H_i)\setminus \{v_i,v_3\}$. The latter shows that $H_i$ cannot be a forest and therefore $H_i$ must have at least one cycle in it. Let $C^{(1)}$ and $C^{(2)}$ be two cycles in $H_1$ and $H_2$ respectively. Let $T$ denote the triangle $v_1u_1u_2$. Note that $V(C^{(1)})\cap V(C^{(2)})\subseteq \{v_3\}$, $V(C^{(1)})\cap V(T)\subseteq \{v_1\}$,$V(C^{(2)})\cap V(T)=\emptyset$. Now, the cycles $C^{(1)}$, $C^{(2)}$ and the triangle $T$ form three cactus-type cycles. 
\end{proof}

\begin{lemma}\label{five_triangles} If $G$ is a $4$-critical planar graph with at least five triangles, then $G$ satisfies $\mathcal{P}$.
\end{lemma}

\begin{proof}
 If there are no two triangles in $G$ sharing a common edge, then $G$ must have three cactus-type cycles, as $G$ has at least five triangles. Now we may assume that there exist two triangles, say $T_1$ and $T_2$, having a common edge. Let $V(T_1)=\{v_1,u,v\}$ and $V(T_2)=\{v_2,u,v\}$. Note that $v_1v_2\notin E(G)$, since $G$ is $4$-critical and hence cannot contain a $K_4$. We may assume that every triangle $T$ in $G$ different from $T_1$ and $T_2$ satisfies the following:

 (i) $T$ does not share an edge with either of the triangles $T_1$ and $T_2$, and

(ii) $T$ does not contain any of the end-vertices of the shared edge $uv$ of $T_1$ and $T_2$

\noindent because if $T$ contains the edge $uv$, then $G$ contains a $K_2\vee \overbar{K}_3$ and the results follow from Lemma~\ref{3triangles_overlap}; if $T$ contains any of the edges $uv_1, uv_2, vv_1, vv_2$, then $G$ contains a $K_1\vee P_4$ and the result  follows from Lemma~\ref{fan_lemma}; if $T$ satisfies (i) and contains any of the vertices $u$ or $v$, then $\Delta (G)\geq 5$.  Let $T_3$ and $T_4$ be two other triangles of $G$ different from $T_1$ and $T_2$. If $T_3$ and $T_4$ are edge-disjoint triangles, then the triangles $T_1,T_3$ and $T_4$ form three cactus-type cycles in $G$.  Now we may assume that the triangles $T_3$ and $T_4$ share a common edge. Let $T_5\notin\{T_1,T_2,T_3,T_4\}$ be another triangle. Repeating the earlier argument, we may assume that $T_5$ does not share an edge with either of the triangles $T_3$ and $T_4$. Now,  $T_1,T_3$ and $T_5$ are three cactus-type cycles.
\end{proof}

\begin{theorem}\label{4crit_prop_P_thm}
Let $G$ be a $4$-critical planar graph different from $K_4$, $M$ and $M^{'}$. Then, $G$ satisfies property $\mathcal{P}$.
\end{theorem}
\begin{proof}
By Gr{\"u}nbaum-Aksenov Theorem, $G$ has at least four triangles. If $G$ has exactly four triangles, then the result follows from Corollary~\ref{lidicky_corollary}. If $G$ has at least five triangles, then the result follows from Lemma~\ref{five_triangles}.
\end{proof}

\section{$4$-chromatic graphs containing the Moser spindle}

\begin{figure}[h!]
\begin{tikzpicture}
[scale=1.25, vertices/.style={draw, fill=black, circle, minimum size = 10pt, inner sep=2pt}, another/.style={draw, fill=black, circle, minimum size = 3.5pt, inner sep=0.1pt}]
\node[another, label=left:{$e$}] (e) at (-1,0) {};
\node[another, label=above:{$a$}] (a) at (0,1) {};
\node[another, label=below:{$a'$}] (a') at (0,.5) {};
\node[another, label=above:{$b$}] (b) at (0,-.5) {};
\node[another, label=below:{$b'$}] (b') at (0,-1) {};
\node[another, label=right:{$d$}] (d) at (1,.25) {};
\node[another, label=below:{$c$}] (c) at (1,-.25) {};
\node[another, label=above:{$u_1$}] (u1) at (1,1) {};
\node[another, label=above:{$u_t$}] (ut) at (2,1) {};
\node[another, label=below:{$v_1$}] (v1) at (1.75,-.25) {};
\node[another, label=below:{$v_k$}] (vk) at (2.75,-.25) {};
\foreach \to/\from in {e/a, e/a', e/b, e/b', a/a', b/b', a/d, a'/d, b/c, b'/c, c/d, a/u1, c/v1}
\draw [-] (\to)--(\from);
\node[label=below:{$\cdots$}](1) at (1.5,1.25){};
\node[label=below:{$\cdots$}](1) at (2.25,0){};
\node[label=above:{$G_1$}](G1) at (0,1.5){};
\end{tikzpicture}
\hspace{.7in}
\begin{tikzpicture}
[scale=1.25, vertices/.style={draw, fill=black, circle, minimum size = 10pt, inner sep=2pt}, another/.style={draw, fill=black, circle, minimum size = 3.5pt, inner sep=0.1pt}]
\node[another, label=left:{$e$}] (e) at (-1,0) {};
\node[another, label=above:{$a$}] (a) at (0,1) {};
\node[another, label=below:{$a'$}] (a') at (0,.5) {};
\node[another, label=above:{$b'$}] (b') at (0,-.5) {};
\node[another, label=below:{$b$}] (b) at (0,-1) {};
\node[another, label=right:{$d$}] (d) at (1,.25) {};
\node[another, label=below:{$c$}] (c) at (1,-.25) {};
\node[another, label=above:{$u_1$}] (u1) at (1,1) {};
\node[another, label=above:{$u_t$}] (ut) at (2,1) {};
\node[another, label=below:{$v_1$}] (v1) at (1,-1) {};
\node[another, label=below:{$v_k$}] (vk) at (2,-1) {};
\foreach \to/\from in {e/a, e/a', e/b, e/b', a/a', b/b', a/d, a'/d, b/c, b'/c, c/d, a/u1, b/v1}
\draw [-] (\to)--(\from);
\node[label=below:{$\cdots$}](1) at (1.5,1.25){};
\node[label=below:{$\cdots$}](1) at (1.5,-.75){};
\node[label=above:{$G_2$}](G2) at (0,1.5){};
\end{tikzpicture}
\caption{Graphs $G_1$, $G_2$ of order $n$ where $0\leq k,t\leq n-7$ and $k+t=n-7$.}
\label{two_diamond_figure}
\end{figure}
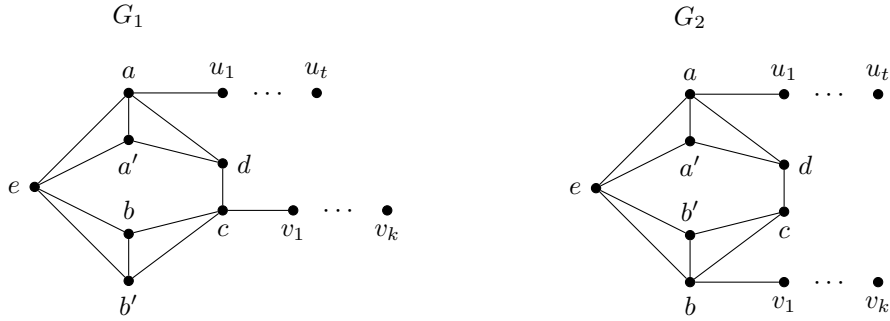

\begin{lemma}\label{G1_lemma} For every $n\geq 7$, 
$\rho(G_1)<\rho (K_4^{(n)})$ where $G_1$ is the graph shown in Fig.~\ref{two_diamond_figure}.
\end{lemma}
\begin{proof}
By Lemma~\ref{godsil}, $\rho (K_4^{(t,k)})\leq \rho (K_4^{(n)})$, so it suffices to show that show that $\rho(G_1)<\rho (K_4^{(t,k)})$. We consider a labelling of the vertices of $K_4^{(t,k)}$ given in Fig.~\ref{G1-comp}.
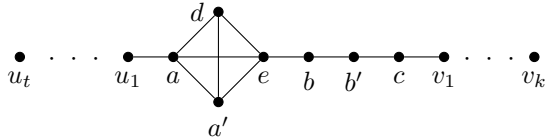
\begin{figure}[h!]
\begin{center}
\begin{tikzpicture}
[scale=1.2, vertices/.style={draw, fill=black, circle, minimum size = 10pt, inner sep=2pt}, another/.style={draw, fill=black, circle, minimum size = 3.5pt, inner sep=0.1pt}]
\node[another, label=left:{$d$}] (d) at (0,1) {};
\node[another, label=below:{$a'$}] (a') at (0,0) {};
\node[another, label=below:{$e$}] (e) at (.5,.5) {};
\node[another, label=below:{$a$}] (a) at (-.5,.5) {};
\node[another, label=below:{$b$}] (b) at (1,.5) {};
\node[another, label=below:{$b'$}] (b') at (1.5,.5) {};
\node[another, label=below:{$c$}] (c) at (2,.5) {};
\node[another, label=below:{$v_1$}] (v1) at (2.5,.5) {};
\node[another, label=below:{$v_k$}] (vk) at (3.5,.5) {};
\node[another, label=below:{$u_1$}] (u1) at (-1,.5) {};
\node[another, label=below:{$u_t$}] (ut) at (-2.2,.5) {};
\foreach \to/\from in {a/a', a/e, a'/e, a/d, a'/d, d/e, e/b, b/b', b'/c, c/v1, a/u1}
\draw [-] (\to)--(\from);
\node[label={$\cdot \  \cdot \  \cdot $}](1) at (3, .2){};
\node[label={$\cdot \  \cdot \  \cdot $}](1) at (-1.6, .2){};
\end{tikzpicture}
\caption{The graph $K_4^{(t,k)}$ with a particular labelling of its vertices to be compared to $G_1$.}
\label{G1-comp}
\end{center}

\end{figure}
 As we pass from $G_1$ to $K_4^{(t,k)}$, the distance between the vertices $d$ and $e$ decrease by $1$ and the distances between all the rest of the vertices increase or stay the same. In particular, the distance between $d$ and $b'$ increases by $1$ and the distance between $d$ and $c$ increases by $3$. Hence, for the Perron vector $x$ of $G_1$, we get
$$x^T\left(D(K_4^{(t,k)})-D(G_1)\right)x>x_d\,\left(x_{b'}+3x_c-x_e\right).$$
Note that  $x_b=x_{b'}$, since $b$ and $b'$ are twin vertices of $G_1$. By  $D(G_1)x=\rho(G_1)x$, 
$$\rho (G_1)\left(x_{b'}+3x_c-x_e\right)=7(x_a+x_{a'})+5x_b-x_c+3x_d+7x_e+\sum_{i=1}^t(3i+7)x_{u_i}+\sum_{i=1}^{k}(3i-1)x_{v_i}.$$ To show that the latter is positive, it suffices to check that $5x_{b}-x_c+7x_e>0$. Using the eigenequation again, we find that $\rho (G_1)(5x_{b}-x_c+7x_e)$ is equal to
$$15(x_a+x_{a'})+17x_b+19x_c+23x_d+3x_e+\sum_{i=1}^t(11i+15)x_{u_i}+\sum_{i=1}^{k}(11i+19)x_{v_i}$$
which is clearly positive. Thus, $x^T\left(D(K_4^{(t,k)})-D(G_1)\right)x>0$ and the result follows.
\end{proof}

\begin{lemma}\label{G2_lemma} For every $n\geq 7$, 
$\rho(G_2)<\rho (K_4^{(n)})$ where $G_2$ is the graph shown in Fig.~\ref{two_diamond_figure}.
\end{lemma}
\begin{proof}
As in Lemma~\ref{G1_lemma}, it suffices to show that $\rho(G_2)<\rho (K_4^{(t,k)})$. This time we consider the labelling of the vertices of $K_4^{(t,k)}$ given in Fig.~\ref{G3-comp}.  

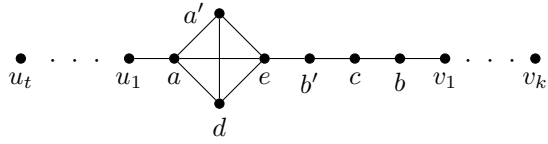
\begin{figure}[h!]
\begin{center}
\begin{tikzpicture}
[scale=1.2, vertices/.style={draw, fill=black, circle, minimum size = 10pt, inner sep=2pt}, another/.style={draw, fill=black, circle, minimum size = 3.5pt, inner sep=0.1pt}]
\node[another, label=left:{$a'$}] (a') at (0,1) {};
\node[another, label=below:{$d$}] (d) at (0,0) {};
\node[another, label=below:{$e$}] (e) at (.5,.5) {};
\node[another, label=below:{$a$}] (a) at (-.5,.5) {};
\node[another, label=below:{$b'$}] (b') at (1,.5) {};
\node[another, label=below:{$c$}] (c) at (1.5,.5) {};
\node[another, label=below:{$b$}] (b) at (2,.5) {};
\node[another, label=below:{$v_1$}] (v1) at (2.5,.5) {};
\node[another, label=below:{$v_k$}] (vk) at (3.5,.5) {};
\node[another, label=below:{$u_1$}] (u1) at (-1,.5) {};
\node[another, label=below:{$u_t$}] (ut) at (-2.2,.5) {};
\foreach \to/\from in {a/a', a/e, a'/e, a/d, a'/d, d/e, e/b', b/b', b'/c, c/v1, a/u1}
\draw [-] (\to)--(\from);
\node[label={$\cdot \  \cdot \  \cdot $}](1) at (3, .2){};
\node[label={$\cdot \  \cdot \  \cdot $}](1) at (-1.6, .2){};
\end{tikzpicture}
\caption{The graph $K_4^{(t,k)}$ with a particular labelling of its vertices to be compared to $G_3$.}
\label{G3-comp}
\end{center}
\end{figure}
As we pass from $G_2$ to $K_4^{(t,k)}$, the distance between the vertices $d$ and $e$ decreases by $1$ and the distances between all the other vertices increase or stay the same. In particular, the distance between the vertices $d$ and $b$, and the distance between $d$ and $c$ increase by $2$. Hence, for the Perron vector $x$ of $G_2$, we get
$$x^T\left(D(K_4^{(t,k)})-D(G_3)\right)x>x_{d}\,\left(2x_{b}+2x_c-x_{e}\right).$$
Now it suffices to check that $\left(2x_{b}+2x_c-x_{e}\right)>0$.  By the eigenequation $D(G_2)x=\rho (G_2)x$, we calculate that $\rho (G_2)(2x_b+2x_c-x_e)$ is equal to
$$7(x_a+x_{a'})+x_b+3x_{b'}+4x_d+6x_e+\sum_{i=1}^k(3i+1)x_{v_i}+\sum_{i=1}^t(3i+7)x_{u_i}$$
which is clearly positive.
\end{proof}

\section {$4$-chromatic graphs containing the Mycielskian of a triangle}

We begin with the graph $M_1^{'}$ in Fig.~\ref{myciel_fig_1} which is obtained from $M^{'}$ by attaching paths at three nonadjacent vertices in  $M^{'}$. Note that the graph $M_1^{'}$ is isomorphic to $M^{'}$ when $r=s=t=1$.
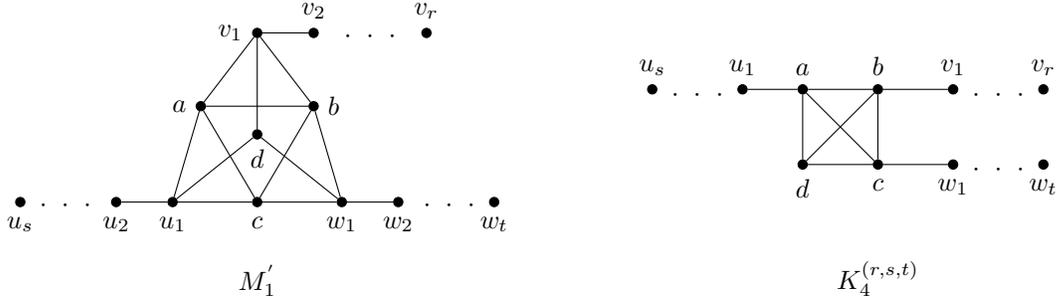
\begin{figure}[h!]
\begin{center}
\begin{tikzpicture}
[scale=1.5, vertices/.style={draw, fill=black, circle, minimum size = 10pt, inner sep=2pt}, another/.style={draw, fill=black, circle, minimum size = 3.5pt, inner sep=0.1pt}]
\node[another, label=left:{$v_1$}] (v1) at (0,1) {};
\node[another, label=above:{$v_2$}] (v2) at (.5,1) {};
\node[another, label=above:{$v_r$}] (vr) at (1.5,1) {};
\node[another, label=left:{$a$}] (a) at (-.5,.35) {};
\node[another, label=right:{$b$}] (b) at (.5,.35) {};
\node[another, label=below:{$d$}] (d) at (0,.1) {};
\node[another, label=below:{$c$}] (c) at (0,-.5) {};
\node[another, label=below:{$u_1$}] (u1) at (-.75,-.5) {};
\node[another, label=below:{$u_2$}] (u2) at (-1.25,-.5) {};
\node[another, label=below:{$u_s$}] (us) at (-2.1,-.5) {};
\node[another, label=below:{$w_1$}] (w1) at (.75,-.5) {};
\node[another, label=below:{$w_2$}] (w2) at (1.25,-.5) {};
\node[another, label=below:{$w_t$}] (wt) at (2.1,-.5) {};
\foreach \to/\from in {a/b, a/v1, b/v1, a/c, b/c, d/v1,d/u1,a/u1,c/u1,c/w1,b/w1,d/w1,v1/v2,w1/w2,u1/u2}
\draw [-] (\to)--(\from);
\node[label={$\cdot \ \cdot \ \cdot$}](1) at (1,.75){};
\node[label={$\cdot \ \cdot \ \cdot$}](1) at (1.7,-.75){};
\node[label={$\cdot \ \cdot \ \cdot$}](1) at (-1.7,-.75){};
\node[label={$M^{'}_1$}](1) at (0,-1.5){};
\end{tikzpicture}
\hspace{.5in}
\begin{tikzpicture}
[scale=2, vertices/.style={draw, fill=black, circle, minimum size = 10pt, inner sep=2pt}, another/.style={draw, fill=black, circle, minimum size = 3.5pt, inner sep=0.1pt}]
\node[another, label=above:{$v_1$}] (v1) at (1,0) {};
\node[another, label=above:{$v_r$}] (vr) at (1.6,0) {};
\node[another, label=above:{$a$}] (a) at (0,0) {};
\node[another, label=above:{$b$}] (b) at (.5,0) {};
\node[another, label=below:{$d$}] (d) at (0,-.5) {};
\node[another, label=below:{$c$}] (c) at (.5,-.5) {};
\node[another, label=above:{$u_1$}] (u1) at (-.4,0) {};
\node[another, label=above:{$u_s$}] (us) at (-1,0) {};
\node[another, label=below:{$w_1$}] (w1) at (1,-.5) {};
\node[another, label=below:{$w_t$}] (wt) at (1.6,-.5) {};
\foreach \to/\from in {a/b,  b/v1, a/c, b/c,a/u1,c/w1,d/a,d/b,d/c}
\draw [-] (\to)--(\from);
\node[label={$\cdot \ \cdot \ \cdot$}](1) at (1.3,-.2){};
\node[label={$\cdot \ \cdot \ \cdot$}](1) at (1.3,-.7){};
\node[label={$\cdot \ \cdot \ \cdot$}](1) at (-.7,-.2){};
\node[label={$K_4^{(r,s,t)}$}](1) at (.5,-1.5){};
\end{tikzpicture}
\caption{The graphs $M^{'}_1$ and $K_4^{(r,s,t)}$ where $ r,s,t\geq 1$ and $r+s+t=n-7$.}
\label{myciel_fig_1}
\end{center}
\end{figure}

\begin{lemma}\label{myciel_lem1} For every $n\geq 7$, $\rho(M_1^{'})<\rho(K_4^{(n)})$.
\end{lemma}
\begin{proof}
Let $\rho =\rho(M_1^{'})$, $\rho '=\rho(K_4^{(r,s,t)})$ and $D=D(M_1^{'})$, $D'=D(K_4^{(r,s,t)})$. By Lemma~\ref{godsil}, $\rho '\leq \rho(K_4^{(n)})$. So, it suffices to show that $\rho < \rho '$. Let us consider a vertex labelling of the graphs as in Fig.~\ref{myciel_fig_1}. As we move from $M_1^{'}$ to $K_4^{(r,s,t)}$ the distance from $d$ to $a$, $b$, or $c$ decreases by one; the distance from $d$ to $u_1$, $v_1$ or $w_1$ increases by one. Moreover, the distances between all the other vertices increase or stay the same. Therefore,
$$\rho ' - \rho\geq x^T(D'-D)x\geq x_d(x_{v_1}+x_{u_1}+x_{w_1}-x_a-x_b-x_c)$$
where $x$ is the Perron vector of $D$. Applying the eigenequation $Dx=\rho x$ repetitively, we find that
\begin{eqnarray*}
\rho (x_{v_1}+x_{u_1}+x_{w_1}-x_a-x_b-x_c) &=& 2(x_a+x_b+x_c)-3x_d\\
&\geq & \frac{1}{\rho}\left(5(x_{u_1}+x_{v_1}+x_{w_1})-2(x_a+x_b+x_c)+12x_d\right)\\ 
&\geq & \frac{1}{\rho ^2}\left(24(x_{u_1}+x_{v_1}+x_{w_1})+40(x_a+x_b+x_c)+3x_d\right)
\end{eqnarray*}
and the latter is clearly positive. Thus, $\rho ' >\rho$.
\end{proof}

\begin{figure}[h!]
\begin{center}
\begin{tikzpicture}
[scale=2, vertices/.style={draw, fill=black, circle, minimum size = 10pt, inner sep=2pt}, another/.style={draw, fill=black, circle, minimum size = 3.5pt, inner sep=0.1pt}]
\node[another, label=left:{$b$}] (b) at (0,0) {};
\node[another, label=above:{$v$}] (v) at (0,.5) {};
\node[another, label=left:{$w$}] (w) at (0,-.5) {};
\node[another, label=right:{$a$}] (a) at (.5,.25) {};
\node[another, label=right:{$c$}] (c) at (.5,-.25) {};
\node[another, label=below:{$u$}] (u) at (1,0) {};
\node[another, label=above:{$d_1$}] (d1) at (1.5,0) {};
\node[another, label=above:{$d_2$}] (d2) at (2,0) {};
\node[another, label=above:{$d_{n-6}$}] (dson) at (2.6,0) {};
\foreach \to/\from in {b/v,b/w,v/a,b/c,b/a,w/c,a/c,u/a,u/c,u/d1,d1/d2}
\draw [-] (\to)--(\from);
\draw [-] (v) to [out=30,in=130,looseness=.7] (d1); 
\draw [-] (w) to [out=330,in=230,looseness=.7] (d1); 
\node[label={$\cdot \ \cdot \ \cdot$}](1) at (2.3,-.2){};
\node[label={$M^{'}_2$}](1) at (1.25,.5){};
\end{tikzpicture}
\hspace{.5in}
\begin{tikzpicture}
[scale=2, vertices/.style={draw, fill=black, circle, minimum size = 10pt, inner sep=2pt}, another/.style={draw, fill=black, circle, minimum size = 3.5pt, inner sep=0.1pt}]
\node[another, label=above:{$a$}] (a) at (0,0) {};
\node[another, label=above:{$c$}] (c) at (.5,0) {};
\node[another, label=below:{$b$}] (b) at (0,-.5) {};
\node[another, label=below:{$w$}] (w) at (.5,-.5) {};
\node[another, label=below:{$v$}] (v) at (1,-.5) {};
\node[another, label=below:{$u$}] (u) at (1.5,-.5) {};
\node[another, label=below:{$d_1$}] (d1) at (2,-.5) {};
\node[another, label=below right:{$d_{n-6}$}] (dson) at (2.6,-.5) {};
\foreach \to/\from in {a/b, a/c,b/c,a/w,c/w,b/w,w/v,v/u,u/d1}
\draw [-] (\to)--(\from);
\node[label={$\cdot \ \cdot \ \cdot$}](1) at (2.3,-.7){};
\node[label={$K_4^{(n)}$}](1) at (1,.3){};
\end{tikzpicture}
\caption{The graphs $M^{'}_2$ and $K_4^{(n)}$.}
\label{myciel_fig_2}
\end{center}
\end{figure}
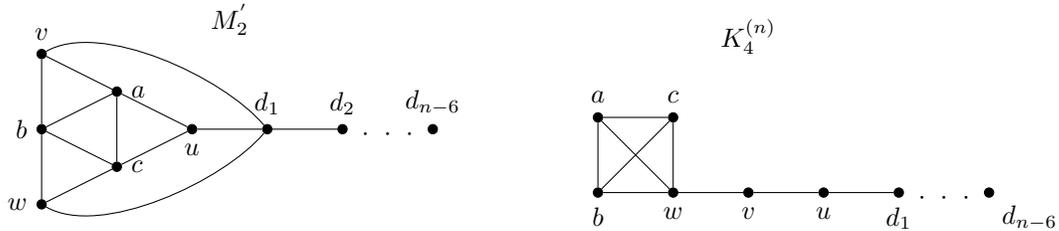

\begin{lemma}\label{myciel_lem2} For every $n\geq 7$, $\rho(M_2^{'})<\rho(K_4^{(n)})$.
\end{lemma}
\begin{proof}
Let $\rho =\rho(M_2^{'})$, $\rho '=\rho(K_4^{(n)})$ and $D=D(M_2^{'})$, $D'=D(K_4^{(n)})$. Let us consider a vertex labelling of the graphs as in Fig.~\ref{myciel_fig_2}. As we move from $M_2^{'}$ to $K_4^{(n)}$ the distances from $w$ to $a$ and $v$ to $u$ decrease by one, the distance from $u$ to $v$ decreases by one, and the distances between all the other vertices increase or stay the same. In particular, the distance from $a$ to $d_1$, $u$ and $v$ increases by $2$, $2$ and $1$ respectively; the distances from $v$ to $b$ and $d_1$ increase by one; and the distance from $c$ to $u$ increases by $2$. Therefore,
$$x^T(D'-D)x\geq x_a(2x_{d_1}+x_v-x_w)+x_u(2x_a+2x_c-x_v)+x_v(x_b+x_{d_1}-x_w)$$
for the Perron vector $x$ of $D$. By the eigenequation $\rho x=Dx$, we have 
$$\rho (2x_{d_1}+x_v-x_w)=3x_a+4x_b+5x_c+2x_u+4x_w+\sum_{j=1}^{n-6}2(j-1)x_{d_j}\,>\, 0,$$
$$\rho (2x_a+2x_c-x_v)=x_a+3x_b+2x_u+6x_v+4x_w+\sum_{j=1}^{n-6}(3j+4)x_{d_j}\, >\, 0,$$
$$\rho (x_b+x_{d_1}-x_w)=x_a+x_b+2x_c+x_u+2x_w+\sum_{j=1}^{n-6}2jx_{d_j}\, >\, 0.$$
Thus, $\rho '-\rho \geq x^T(D'-D)x >0 .$
\end{proof}

\section{Proof of Theorem~\ref{main}}

We are now ready to prove our main result.\\

\noindent {\it Proof of Theorem~\ref{main}.} Let $G$ be an extremal graph with maximal distance spectral radius among all connected $4$-chromatic planar graphs of order $n$. By the proof Lemma~\ref{k_critical_lemma}, $G$ is obtained from a $4$-critical planar graph $H$ by attaching paths at vertices of some independent set $S$ of $H$. If a nontrivial path is attached at a degree $4$ vertex in $S$, then $\Delta (G)\geq 5$ and Lemma~\ref{Delta5} shows that $G$ cannot be an extremal graph. So, we may assume that $S$ consists of nonadjacent vertices of degree at most three. Now we shall consider three cases:

\noindent {\it Case 1:} $H$ is $K_4$. The complete graph $K_4$ has no two non-adjacent vertices. So, $G$ must be $K_4^{(n)}$ if $H$ is $K_4$.
 
\noindent {\it Case 2:} $H$ is the Moser spindle $M$. Consider a vertex labelling of $M$ as in Fig.~\ref{diam_exp_1}. The graph $M$ has eight independent sets consisting of vertices of degree at most three.  However, by symmetry, it suffices to consider the independent sets $\{a,c\}$ and $\{a,b\}$ only. In Lemmas~ \ref{G1_lemma} and \ref{G2_lemma} it was shown that the distance spectral radius of a graph obtained from $M$ by attaching two paths at the vertices in $\{a,c\}$ or $\{a,b\}$ respectively is less than that of $K_4^{(n)}$. Therefore, this case is not possible.

\noindent {\it Case 3:} $H$ is the Mycielskian of a triangle $M^{'}$. Using the symmetry and the fact that $S$ does not contain a degree $4$ vertex of $M^{'}$, it suffices to consider the graphs $M_1^{'}$ and $M_2^{'}$ shown in Figures~\ref{myciel_fig_1} and \ref{myciel_fig_2}. It follows from Lemmas~\ref{myciel_lem1} and \ref{myciel_lem2} that $G$ cannot be an extremal graph in this case.

\noindent {\it Case 4:} $H$ is different from $K_4$, $M$ and $M^{'}$. By Theorem~\ref{4crit_prop_P_thm}, $H$ satisfies property $\mathcal{P}$ and therefore $G$ also satisfies $\mathcal{P}$. By Corollary~\ref{property_P}, the graph $G$ cannot be an extremal graph.

Thus, $H$ must be $K_4$ and the unique extremal graph is $K_4^{(n)}$.
\qed

\vskip0.3in

\bibliographystyle{elsarticle-num}

\end{document}